\documentclass[11pt]{amsart}
\usepackage{amsmath, amssymb, mathrsfs,graphicx}
\usepackage{mathabx}
\usepackage{bbm} 

\usepackage{tikz}
\usetikzlibrary{patterns}

\newcommand{\executeiffilenewer}[3]{%
\ifnum\pdfstrcmp{\pdffilemoddate{#1}}%
{\pdffilemoddate{#2}}>0%
{\immediate\write18{#3}}\fi%
}

\newcommand{\abs}[1]{{\left\lvert{#1}\right\rvert}}
\newcommand{\norm}[2][]{{\left\lVert{#2}\right\rVert}_{#1}}
\newcommand{\ang}[1]{{\left\langle{#1}\right\rangle}}
\newcommand{\smallabs}[1]{{\lvert{#1}\rvert}}

\newcommand{\hamvf}{\mathsf{H}}
\renewcommand{\Im}{\operatorname{Im}}
\renewcommand{\Re}{\operatorname{Re}}

\DeclareMathOperator{\supp}{supp}

\newcommand{\CI}{\mathcal{C}^\infty}
\newcommand{\CcI}{\mathcal{C}_c^\infty}

\newcommand{\pa}{{\partial}}
\newcommand{\ep}{{\epsilon}}

\DeclareMathOperator{\id}{I}
\newcommand{\RR}{\mathbb{R}}
\newcommand{\reals}{\mathbb{R}}

\newcommand{\complexes}{\mathbb{C}}
\newcommand{\integers}{\mathbb{Z}}

\renewcommand\Re{\operatorname{Re}}
\newcommand{\module}{\mathcal{M}}
\newcommand{\algebra}{\mathcal{A}}
\newcommand{\hilbert}{\mathcal{H}}
\newcommand{\coiso}{\mathcal{I}}
\newcommand{\tcoiso}{\widetilde{\mathcal{I}}}

\newcommand{\Lap}{\Delta}

\DeclareMathOperator{\WF}{WF}

\newtheorem{assumption}{Assumption}
\newcommand{\flowout}{\mathcal{F}}
\newcommand{\flowoutX}{{}^{X}\!\mathcal{F}}
\newcommand{\Id}{I}

\newcommand{\dom}{\mathcal{D}}

\newcommand{\energy}{\mathcal{E}}
\newcommand{\residual}{\mathcal{R}}
\newcommand{\outgoing}{\mathcal{O}}

\DeclareMathOperator{\Res}{Res}
\newcommand{\dtilde}[1][t]{{\stackrel{D,#1}{\sim}}}
\newcommand{\gtilde}[1][t]{{\stackrel{G,#1}{\sim}}}

\newcommand{\Tbstar}{{}^bT^*}
\newcommand{\Sbstar}{{}^bS^*}
\newcommand{\loc}{\text{loc}}

\newtheorem{theorem}{Theorem}
\newtheorem{lemma}{Lemma}
\newtheorem{proposition}{Proposition}
\newtheorem{corollary}{Corollary}

\theoremstyle{remark}
\newtheorem{definition}{Definition}
\newtheorem{remark}{Remark}

\newcommand{\pd}[1][]{\partial_{#1}}

\newcommand{\Ut}[1][t]{\mathcal{U}(#1)}
\newcommand{\Utzero}[1][t]{\mathcal{U}_0(#1)}
\newcommand{\PD}[1][]{D_{#1}}

\title[Resolvent estimates on conic manifolds]{Resolvent estimates and
  local decay of waves on conic manifolds} 
\author{Dean Baskin and Jared Wunsch} \address{Northwestern University}
\thanks{The authors are grateful to Euan Spence and to Maciej Zworski
  for illuminating discussions.  The first author acknowledges the
  support of NSF postdoctoral fellowship DMS-1103436, and the second
  author was partially supported by NSF grant DMS-1001463.}

\date{\today}

\begin{document}
\maketitle

\begin{abstract}
  We consider manifolds with conic singularites that are isometric to
  $\reals^{n}$ outside a compact set.  Under natural geometric
  assumptions on the cone points, we prove the existence of a logarithmic
  resonance-free region for the cut-off resolvent.  The estimate also
  applies to the exterior domains of non-trapping polygons via a
  doubling process.

  The proof of the resolvent estimate relies on the propagation of
  singularities theorems of Melrose and the second
  author~\cite{Melrose-Wunsch:cone} to establish a ``very weak''
  Huygens' principle, which may be of independent interest.

  As applications of the estimate, we obtain a exponential local energy
  decay and a resonance wave expansion in odd dimensions, as well as a
  lossless local smoothing estimate for the Schr{\"o}dinger equation.
\end{abstract}

\section{Introduction}

In this paper we consider a manifold $X$ of dimension $n$ with conic singularities that is
isometric to $\RR^n$ outside a compact set.  We impose geometric hypotheses
(elucidated in Section~\ref{section:assumptions} as Assumptions~1--3) that \begin{enumerate}
\item The flow
along ``geometric'' geodesics is non-trapping.  (Geometric geodesics
are those that miss the cone points or that are everywhere locally
given by limits of families of geodesics missing the cone points.)
\item 
No three cone points are collinear.
\item
No two cone points are conjugate to each other.
\end{enumerate}
Our main result is as follows (throughout the paper, $\Lap$ denotes
the Laplacian with \emph{positive} spectrum):
\begin{theorem}\label{theorem:resest}
For $\chi \in \CcI(X),$ there exists $\delta>0$ such that the cut-off resolvent
$$
\chi (\Lap-\lambda^2)^{-1}\chi
$$
can be analytically continued from $\Im \lambda >0$ to the region
$$
\Im \lambda >-\delta \log \Re \lambda,\ \Re \lambda >\delta^{-1}
$$
and for some $C,T>0$ enjoys the estimate
$$
\norm[L^{2}\to L^{2}]{\chi (\Lap-\lambda^2)^{-1}\chi} \leq C
\smallabs{\lambda}^{-1} e^{T\smallabs{\Im \lambda}}
$$
in this region.
\end{theorem}
As shown by Lax-Phillips \cite{Lax-Phillips1} and Vainberg
\cite{Vainberg} in certain geometric settings and later generalized by
Tang-Zworski \cite{TZ} to ``black-box'' perturbations, if the dimension $n$ is odd, then
Theorem~\ref{theorem:resest} results in a decay estimate for solutions to
the wave equation in such a geometry, and indeed in a full resonance-wave
expansion for solutions to the wave equation.  Let $\dom_s$ denote the
domain of $\Lap^{s/2}$ (see Section~\ref{section:geometry} below) and let
$\sin t\sqrt\Lap/\sqrt\Lap$ be the wave propagator.  Let $\chi$ equal $1$
on the set where $X$ is not isometric to $\RR^n.$
\begin{corollary}\label{corollary:resexp}
Let $n$ be odd.  For all $A>0,$
small $\ep>0,$ $t>0$ sufficiently large,  and $f \in
\dom_1,$
$$
\chi \frac{\sin t\sqrt{\Lap}}{\sqrt{\Lap}} \chi f=  \sum_{\substack{\lambda_j \in \Res(\Lap) \\ \Im \lambda
  >-A}}\sum_{m=0}^{M_j} e^{-it \lambda_j} t^m w_{j,m} + E_A(t) f
$$
where the sum is of resonances of $\Lap,$ i.e.\ over the poles of the
meromorphic continuation of the resolvent, and the $w_{j,m}$ are the
associated resonant states corresponding to $\lambda_j.$
The error satisfies
$$
\norm[\dom_{1}\to L^{2}]{E_A(t)} \leq C_\ep e^{-(A-\ep)t}.
$$

In particular, since the resonances have imaginary part bounded above by a
negative constant, $\chi \frac{\sin t\sqrt{\Lap}}{\sqrt{\Lap}} \chi f$ is
exponentially decaying.
\end{corollary}
(We refer the reader to Theorem~1 of \cite{TZ} for details of the resonance
wave expansion.)

Another consequence of our resolvent estimate is a \emph{local smoothing}
estimate without loss for the Schr{\"o}dinger equation.  Local smoothing
estimates were originally established for the Schr{\"o}dinger equation on
$\reals^{n}$ by Sj{\"o}lin~\cite{Sjolin}, Vega~\cite{Vega},
Constantin--Saut~\cite{CS}, Kato--Yajima~\cite{KYajima}, and
Yajima~\cite{Yajima}.  Doi~\cite{Doi} showed that on smooth manifolds the
absence of trapped geodesics is necessary for the local smoothing estimate
to hold without loss.  We show that even in the presence of very weak
trapping due to the diffractive geodesics, the local smoothing estimate holds
without loss.
\begin{corollary}
  \label{corollary:local-smoothing}
  Suppose $u$ satisfies the Schr{\"o}dinger equation on $X$:
  \begin{align*}
    i^{-1}\pa_t u(t,z) + \Lap  u(t,z) &= 0 \\
    u(0,z) &= u_{0}(z)\in L^{2}(X)
  \end{align*}
  Then for all $\chi \in C^{\infty}_{c}(X)$, $u$ satisfies the local smoothing estimate without loss:
  \begin{equation*}
    \int_{0}^{T}\norm[\mathcal{D}^{1/2}]{\chi u(t) }^{2}\, dt \leq
    C_{T} \norm[L^{2}(X)]{u_{0}}^{2}
  \end{equation*}
\end{corollary}
This result follows directly from our Theorem~\ref{theorem:resest} by an
argument of Burq~\cite{Burq:smoothing}.

Another application of the resolvent estimate of
Theorem~\ref{theorem:resest} is to the damped wave equation.  Although
we do not pursue it here, under suitable convexity assumptions (e.g.,
if no geodesic passing through the perturbed region re-enters it; see
Datchev--Vasy~\cite{DV1,DV2} for more general conditions), it is
possible to obtain decay estimates for the damped wave equation on
conic manifolds when the only undamped geodesics are diffractive ones.
This relies on a gluing construction of Datchev--Vasy to obtain a
suitable resolvent estimate and on the recent work of Christianson,
Schenck, Vasy and the second author~\cite{CSVW} to yield the estimate.

In addition to applying to manifolds with cone points, our results
also apply to the more elementary setting of certain exterior domains
to polygons in the plane.  Let $\Omega\subset \reals^{2}$ be a compact
region with piecewise linear boundary.  We further suppose that the
complement $\reals^{2}\setminus \Omega$ is connected, that no three
vertices of $\overline{\Omega}$ are collinear, and that
$\reals^{2}\setminus \Omega$ is non-trapping, in the sense that all
billiard trajectories not passing through the vertices of
$\overline{\Omega}$ escape to infinity.\footnote{In fact we require a
  slightly different condition.  We ask that all billiard trajectories
  that are \emph{locally} approximable by trajectories missing the
  vertices escape to infinity.  This is not quite the same condition
  but is generically equivalent.}  Figure~\ref{fig:polygon-example}
illustrates an example of such an exterior domain.  For this class of
domains, the analogue of Theorem~\ref{theorem:resest} holds.
\begin{figure}[htp]
  \centering
  \begin{tikzpicture}
    \filldraw[pattern=north west lines, pattern color = red] (-1,-1)
    -- (-3,-2) -- (-3,1) -- (-1,1) -- (-1,-1);
    \filldraw[pattern=north west lines, pattern color = red] (1,0) --
    (4, 1) -- (3, -2) -- (1,0);
    \draw[dashed] (-1,0) -- (1,0);
    \draw[dashed, ->] (-1,0) -- (-0.5,0);
    \draw[dashed, ->] (1,0) -- (0.5,0);
  \end{tikzpicture}
  \caption{An example of a domain to which
    Corollary~\ref{corollary:polygon} applies.  The dashed line
    represents a trapped diffractive orbit.}
  \label{fig:polygon-example}
\end{figure}
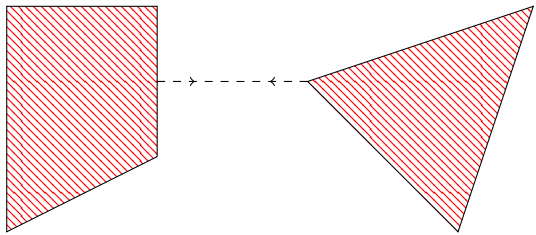
\begin{corollary}
  \label{corollary:polygon}
  If $X = \reals^{2}\setminus \Omega$ is the exterior of a
  non-trapping polygon with no three vertices collinear and 
  $\Lap$ is the Dirichlet or Neumann extension of the Laplacian
  on $X$, then the result of Theorem~\ref{theorem:resest} holds for
  the resolvent on $X$.  
\end{corollary}
The proof of Corollary~\ref{corollary:polygon} relies on reducing the
problem to one on a surface with conic singularites.  Indeed, such an
exterior domain can be \emph{doubled} by gluing together two copies of
it across the common boundary; this results in a manifold with cone
points, corresponding to the vertices of the initial polygonal domain,
and with two ends each isometric to $\RR^2.$ Solutions to the wave
equation in this ``doubled'' manifold are closely related to solutions
to the Dirichlet or Neumann problem on the original exterior domain
via the method of images. Our results hold for the exterior problem to
such non-trapping polygons as well, although this entails some mild
complication in the proof (the introduction of ``black-box''
methods)---see Section~\ref{sec:exter-polyg-doma} below.  In
particular, our result affirmatively answers a conjecture of
Chandler-Wilde, Graham, Langdon, and Spence~\cite{CWGLS:2012}.  In the
case when the obstacle is star-shaped, we remark that the exponential
energy decay is a consequence of the classical technique of Morawetz
estimates (see e.g. Lemma 3.5 of \cite{CWM})); we believe that the
estimate for general non-trapping polygons is new, however.

We note that stronger estimates than those of Theorem~\ref{theorem:resest}
are known to hold in the case of a non-trapping metric or even an
appropriately non-trapping ``black box'' perturbation
such as a smooth non-trapping obstacle  (see \cite{Lax-Phillips1},
\cite{Vainberg}, \cite{TZ}): in these cases there are finitely many
resonances above \emph{any} logarithmic curve $\Im \lambda>-N \log \Re
\lambda.$ That the result here is likely to be sharp can be seen from the
explicit computation of Burq \cite{Burq:coins}, who shows that in the case
of obstacle scattering by two strictly convex analytic obstacles in
$\RR^2,$ one of which has a corner, the resonances (poles of the analytic
continuation of the resolvent $(\Lap-\lambda^2)^{-1}$) are located along
curves of the form $\Im \lambda =-C \log \Re \lambda.$ Burq's setting is
not of course exactly that of manifolds with cone points, but is
suggestively close to that of polygonal domains discussed above.  (Similar
logarithmic strings of resonance poles also appear in
Zworski~\cite{Zworski4} where they are shown to arise from finite order
singularities of a one-dimensional potential, substantiating heuristics
from Regge \cite{Regge:Analytic}.)

By contrast, what seems the weakest trapping possible in the setting of
smooth manifolds, a single closed hyperbolic geodesic, is known in certain
settings to produce strings of resonances along lines of constant imaginary
part \cite{Colin-Parisse}, and hence yields an analytic continuation to a
smaller region than that shown here, which does not permit a resonance wave
expansion in the strong sense of Corollary~\ref{corollary:resexp} except in
very special cases \cite{CZ}.

The fact that the estimates demonstrated here are weaker, by only a very
small margin, than those for non-trapping situations, reflects that fact
that cone points induce a kind of ``weak trapping:'' there exist geodesics
connecting every pair of cone points, and concatenation of such geodesics
starting and ending at the same cone point should be considered a
legitimate geodesic curve in a conic geometry.  In particular, such
concatenations of geodesics are known (generically) to propagate
singularities of the wave equation on exact cones by results of
Cheeger-Taylor \cite{Cheeger-Taylor1, Cheeger-Taylor2}; this reflects the
\emph{diffraction} of singularities by the cone point.  Melrose and the
second author \cite{Melrose-Wunsch:cone} subsequently showed that on any
manifold with conic singularities, the propagation of singularities is
limited to geodesics entering and leaving a given cone point at the same
time (``diffractive propagation'').  It was further shown in
\cite{Melrose-Wunsch:cone} that the fundamental solution of the wave group
with initial pole near a cone point was \emph{smoother} along generic
geodesics emerging from the cone point than along those that are
approximable by geodesics emanating from the initial pole and missing the
cone point; this ``smoothing effect'' in fact holds for any solution that satisfies
an appropriate \emph{nonfocusing} condition with respect to the cone point
in question (see Section~\ref{section:MW} below).  Thus,
colloquially, \cite{Melrose-Wunsch:cone} showed that ``diffracted
singularities are smoother than geometrically propagated singularities.''
It also showed that the spherical wavefront of diffracted singularities is
a conormal wave.  It is the smoothing property and the conormality that
play an essential role in the proof of Theorem~\ref{theorem:resest}.  The
proof proceeds via another result which may be of independent interest, a
theorem on the \emph{weak non-trapping of singularities} for manifolds with
cone points.  In the following theorem, $U(t)$ denotes the wave group, and
$\energy_r$ denotes the Sobolev space of energy data $\dom_r\oplus \dom_{r-1}.$
\begin{theorem}\label{theorem:Huygens}
Let $\chi \in \CcI(X).$   For any $s \in \RR,$ there exists $T_s\gg 0$
such that whenever $t>T_s,$
$$
\chi U(t) \chi: \energy_r\to \energy_{r+s}
$$
for all $r.$
\end{theorem}
We recall that Huygens' Principle, valid in odd dimensional Euclidean
space, says that $\chi U(t)\chi$ is eventually identically zero.  More
generally, in even dimensional Euclidean space or indeed in any
``non-trapping'' metric in which all geodesics escape to infinity,
$\chi U(t) \chi$ eventually has a Schwartz kernel in $\mathcal{C}^\infty$.  Our
Theorem~\ref{theorem:Huygens} is weaker yet: here the cut-off wave kernel is as
smooth as one likes, after a sufficiently long time.

Resolvent estimates similar to ours have been previously demonstrated by
Duyckaerts \cite{Duyckaerts} for operators of the form $\Lap+V$ where
$\Lap$ is the Euclidean Laplacian and $V$ has multiple inverse-square
singularities: these singularities are analytically similar to (albeit
geometrically simpler than) cone points.

\section{Geometric set-up}\label{section:geometry}
The basic material in this section on conic geometry comes from
\cite{Melrose-Wunsch:cone} while the more detailed discussion of the
global geometry of geodesics is taken from \cite{Wunsch:Poisson}.

Let $X$ be a noncompact manifold with boundary, $K$ a compact subset
of $X$, and let $g$ be a Riemannian metric on $X^\circ$ such that
$X\backslash K$ is isometric to the exterior of a Euclidean ball
$\RR^n\backslash \overline{B^n(0,R_0)}$ and such that $g$ has conic
singularities at the boundary of $X$:
$$
g=dx^2+ x^2 h(x,dx,y,dy);
$$
here $g$ is assumed to be nondegenerate over $X^\circ$ and $h |_{\pa X}$
induces a metric on $\pa X.$ We let $Y_\alpha,$ $\alpha=1,\dots N$ denote
the components of $\pa X;$ we will refer to these components in what
follows as \emph{cone points}, as each boundary component is a single point
when viewed in terms of metric geometry.

We further let
$$
M=\RR\times X
$$
denote our space-time manifold.

We recall from Theorem~1.2
of \cite{Melrose-Wunsch:cone} that by judicious choice of coordinates
$x,y$ on a collar neighborhood of $\pa X,$ we may reduce $g$ to the
normal form
\begin{equation}\label{semiproduct}
g=dx^2+ x^2 h(x,y,dy),
\end{equation}
where $h$ is now a family (in $x$) of metrics on $Y.$  Then the
curves $y=\text{const.}$ are geodesics, with $x$ the length
parameter.  Indeed, the curves of this form are the \emph{only} geodesics
reaching $\pa X$ and they foliate a neighborhood of $\pa X.$  We let
$\flowoutX^s_\alpha$ denote the collection of the continuations of
forward and backward bicharacteristics in $T^*X^\circ$ which reach the
boundary component $Y_\alpha$ in time $\abs{t}\leq s$ (with $\flowout$
denoting ``flowout'' of the cone point $Y_\alpha$).  Thus for small $s,$ in
canonical coordinates $\xi,\eta$ dual to $x,y,$
$$
\flowoutX^s_\alpha = U\cap \{x\in (0,s), y \in Y_\alpha, \xi\in \RR, \eta=0\}
$$
where $U$ is a neighborhood of the single boundary component
$Y_\alpha$ containing a component of $x<s.$ 
We further refer to points in $\flowoutX_\alpha^s$ as \emph{incoming}
or \emph{outgoing} with respect to the cone point according to whether
they reach the boundary at positive or negative time respectively
under the flow (this separates the manifold into components).  We will
also be concerned with the corresponding flowout sets in space-time.
Letting $\Sigma$ denote the characteristic set of $\Box = D_{t}^{2} -
\Lap$ on $T^*M^\circ$ we define
$$
\flowout^s_\alpha=\{(t , \tau, z,\zeta) \in \Sigma:\
(z,\zeta) \in \flowoutX^s_\alpha\}\subset T^*M^\circ.
$$
 As discussed in \cite{MVW:edges}
(where the notation $\flowout$ was first used) the manifolds
$\flowoutX_\alpha^s,$ $\flowout_\alpha^s$ are \emph{coisotropic} conic
submanifolds of $T^*X^\circ$, resp.\ $T^*M^\circ.$ 

We let $\dom_s$ denote the domain of the $s/2$ power of the
Friedrichs extension of the Laplacian on $\mathcal{C}_c^\infty
(X^\circ).$  Note that this agrees with the ordinary Sobolev space
$H^s$ away from the cone points (and was characterized in
\cite[Section 3]{Melrose-Wunsch:cone} in terms of the scale of \emph{weighted
b-Sobolev spaces}).  Let $$\energy_s=\dom_s\oplus \dom_{s-1}$$ denote the
corresponding space of Cauchy data for the wave equation, and let
$$
U(t) =\exp it \begin{pmatrix} 0 & \Id \\ \Lap & 0\end{pmatrix}
$$
denote the wave propagator, hence
$$
U(t)\colon \energy_s \to \energy_s
$$
for each $s \in \RR.$
We will frequently need to deal with error terms that are residual in the
scale of space $\energy_s,$ so we define
$$
\residual=\big\{ R: \energy_{-\infty} \to \energy_{+\infty,c}\big\}
$$
with the additional $c$ subscript denoting compact support in $X$.  In
dealing with wave equation solutions as functions in spacetime, it is
convenient to think of them lying in the Hilbert space
$$
L^2([0, \widetilde T]; \energy_s)
$$
with $\widetilde T\gg T_s$ taken large enough to encompass all time
intervals under consideration.  We thus denote this space
$$
L^2\, \energy_s
$$
for brevity.  We recall that solutions to the wave equation in $L^2\,
\energy_s$ have unique restrictions to fixed-time data lying in $\energy_s,$
and will use this fact freely in what follows.

For convenience, we will equip the cosphere bundle $S^*X$ with a
Riemannian metric inducing a distance function, denoted $d(\bullet,
\bullet).$ 

\subsection{Geometric and diffractive geodesics}\label{subsection:bichars}

We now recall the different notions of ``geometric'' and
``diffractive'' bicharacteristic which enter into the propagation of
singularities on manifolds with cone points.

\begin{definition}
A \emph{diffractive geodesic} on $X$ is a union of a finite number of
closed, oriented geodesic segments $\gamma_1, \dots, \gamma_N$ in $X$
such that all end points except possibly the initial point in $\gamma_1$ and the
final point of $\gamma_N$ lie in $Y=\pa X$, and $\gamma_i$ ends at the same
boundary component at which $\gamma_{i+1}$ begins, for $i=1, \dots, N-1$.

A \emph{geometric geodesic} is a diffractive geodesic such that in
addition, the final point of $\gamma_i$ and the initial point of
$\gamma_{i+1}$ are connected by a geodesic of length $\pi$ in a boundary
component $Y_\alpha$
(w.r.t.\ the metric $h_0=h\rvert_{Y_\alpha}$) for $i=1,\dots, N-1$.
\end{definition}

The proof the following proposition was sketched in
\cite{Melrose-Wunsch:cone}, and yields the equivalence of the above
definition of ``geometric geodesic'' with the more casual one used in the
introduction above:
\begin{proposition} The geometric geodesics
are those that are \emph{locally} realizable as limits of families
of geodesic in $X^\circ$ as they approach a given boundary component.
\end{proposition}
\begin{figure}
\includegraphics[scale=0.2,angle=90]{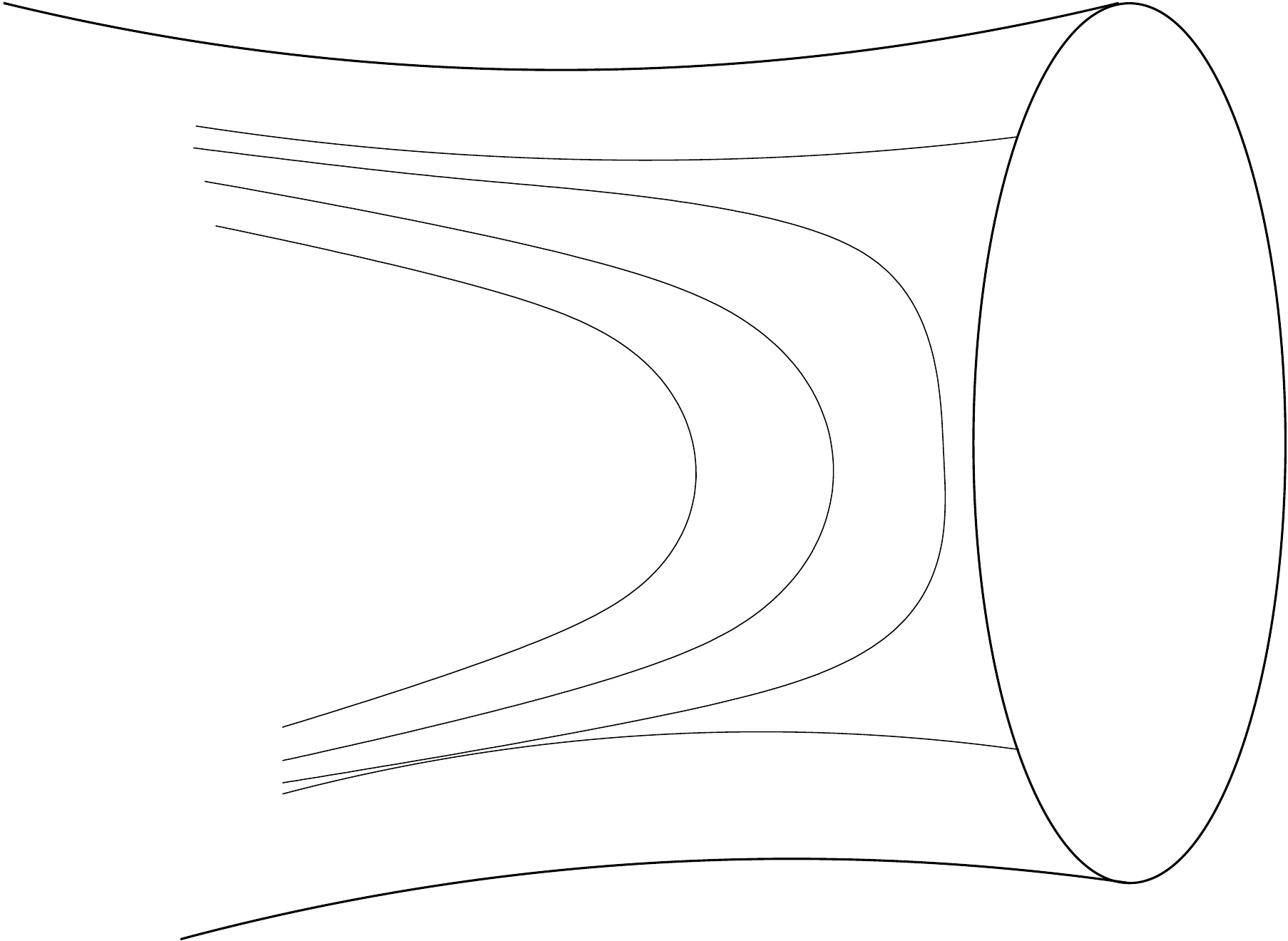}
\caption{A family of geodesics missing the cone point (i.e., boundary
  component), limiting to a pair of geodesics entering and leaving it
  normally together with a geodesic of length $\pi$ in the boundary
  connecting the two.}
\end{figure}
\begin{remark}\label{remark:nonapproximable}
  We note that while every geometric geodesic is \emph{locally}
  approximable by smooth geodesics in $X^\circ,$ a geodesic undergoing
  multiple interactions with cone points may not be \emph{globally}
  approximable in this sense.  Figure~\ref{figure:nonapproximable} shows
  such a situation.  It is most easily interpreted as showing a domain with
  boundary given by the three slits; geodesics then reflect off smooth
  parts of the boundary, and geometric geodesics either pass straight
  through the end of the slits or reflect specularly as if the slits
  continued.  Then the vertical line, while it can be uniformly approximated by
  broken geodesics as shown, cannot be approximated globally: any
  approximating geodesic would have to reflect off one or another of the
  slits.  To place this example in the context of the current paper rather
  than that of domains with boundary, we should instead interpret the
  picture as showing one sheet of a two-sheeted ramified cover of $\RR^2,$
  with the slits representing branch cuts.  This makes the ends of the
  slits into cone points, with the link of each cone point a circle of
  circumference $4\pi.$ In this situation, any unbroken approximating
  geodesic would have to move onto the other sheet of the cover by passing
  through one of the slits, hence could not globally approximate the line
  shown, which remains on a single sheet of the cover.

  This very simple example has three collinear cone points, which we are
  ruling out by hypothesis; however one can make other examples involving
  only two interactions with cone points, and this is permitted by our
  geometric hypotheses.  For this reason Assumption~\ref{assumption:escape}
  is formulated so as to cover geometric geodesics explicitly, rather than
  just as a uniform statement on geodesics in $X^\circ$ (which might be
  preferable).  Ruling out propagation along non-approximable geometric
  geodesics, if indeed true, would likely require somewhat delicate
  second-microlocal arguments.
\begin{figure}
\centering
\def\svgwidth{2in}
\executeiffilenewer{nonapproximable.svg}{nonapproximable.pdf}%
{inkscape -z -D --file=nonapproximable.svg %
--export-pdf=nonapproximable.pdf --export-latex}%
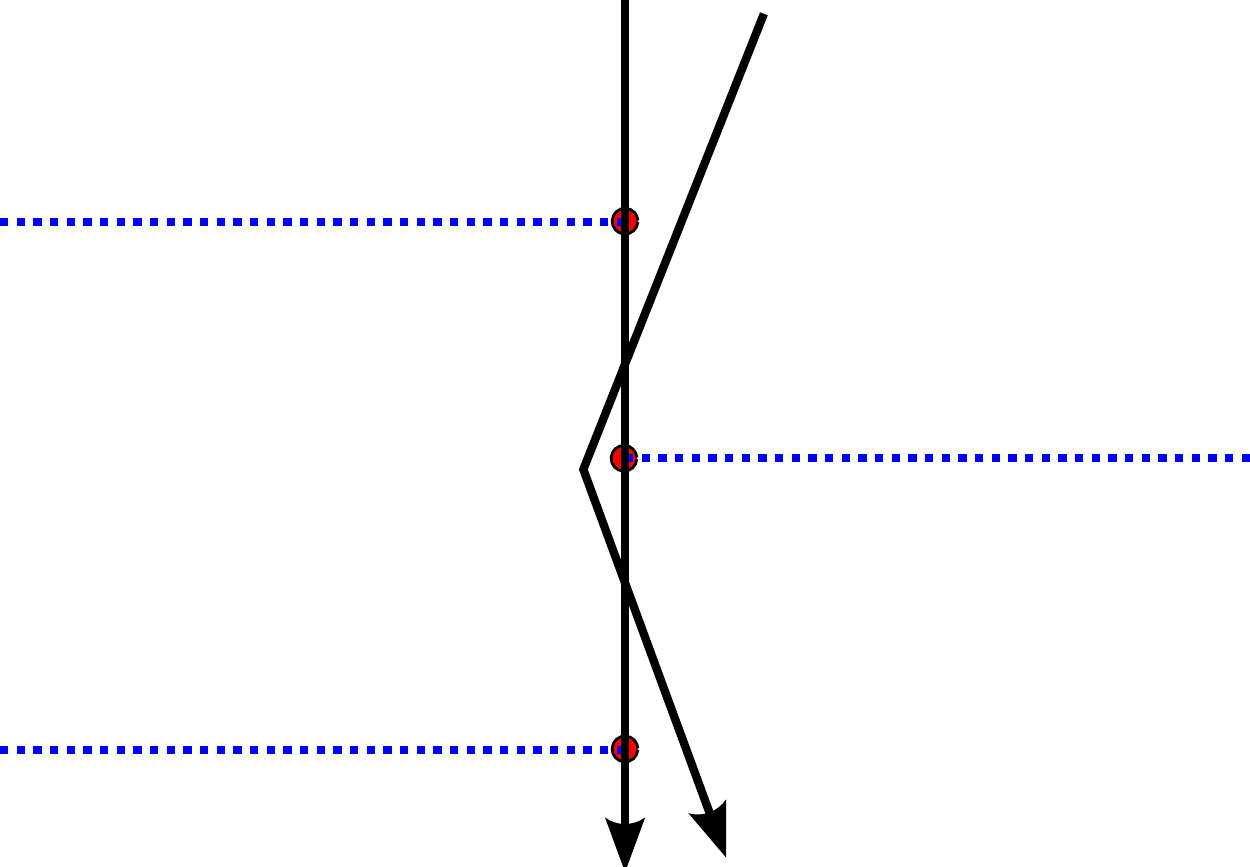%

\caption{A non-approximable geometric geodesic in the
  triply-slit-plane (or its branched cover), with an approximating
  broken geodesic in $X^\circ.$}\label{figure:nonapproximable}
\end{figure}

\end{remark}

As we will need to understand the lift of the geodesic flow to the
cotangent bundle, it is helpful to see how this can be accomplished
uniformly up to $\pa X$ (although in this paper, microlocal considerations
will only arise over $X^\circ,$ which is a considerable simplification).
We let $\Tbstar X$ denote the b-cotangent bundle of $X$, i.e.\ the dual of
the bundle whose sections are smooth vector fields tangent to $\pa X.$   Let $\Sbstar X$ denote the
corresponding sphere bundle.  Let $\xi \, dx/x+ \eta \cdot dy$ denote the
canonical one-form on $\Tbstar X$. (We refer the reader to
Chapter 2 of \cite{Melrose:APS} for a further explanation of ``b-geometry,''
of which we only use the rudiments here.)

Let $K_g$ be the Hamilton vector field (with respect to the symplectic
form $d\left( \xi\, dx/x + \eta \cdot dy\right)$) for
$g/2=(\xi^2+h(x,y,\eta))/(2x^2)$, the symbol of $\Lap/2$ on $\Tbstar
X;$ note that $K_g$ is merely the geodesic spray in $\Tbstar X$ with
velocity $\sqrt g.$ It is convenient to rescale this vector field so
that it is both tangent to the boundary of $X$ and homogeneous of
degree zero in the fibers. Near a boundary component $Y_\alpha$, for a
metric in the reduced form \eqref{semiproduct}, we have (see
\cite{Melrose-Wunsch:cone})
\begin{equation}
K_g=x^{-2} \bigg( H_{Y_\alpha}(x) + \big(\xi^2+h(x,y,\eta)+ \frac x2 \frac{\pa
h}{\pa x}\big) \pa_\xi +\xi x \pa _x \bigg),
\label{spray}
\end{equation}
where $H_{Y_\alpha}(x)$ is the geodesic spray in $Y_\alpha$ with respect to the
family of metrics $h(x, \cdot)$.  Hence the desired rescaling is
$$
Z=\frac{x}{\sqrt g} K_g.
$$ (Note here that $g$ refers to the \emph{metric function} on the
cotangent bundle and not the determinant of the metric tensor.)  By the
homogeneity of $Z,$ if we radially compactify the fibers of the
cotangent bundle and identify $\Sbstar X$ with the ``sphere at
infinity'' then $Z$ is tangent to $\Sbstar X,$ and may be restricted
to it.  Henceforth, then, we let $Z$ denote the restriction of
$(x/\sqrt g) K_g$ to the \emph{compact} manifold $\Sbstar X$ on which
the coordinates $\xi, \eta$ have been replaced by the (redundant)
coordinates $$(\bar\xi, \bar\eta)=\left(\xi/\sqrt{\xi^2+h(\eta)},
  \eta/\sqrt{\xi^2+h(\eta)}\right).$$ $Z$ vanishes only at certain
points $x=\bar\eta=0$ over $\pa X$, hence the closures of maximally
extended integral curves of this vector field can only begin and end
over $\pa X$.  Since $Z$ is tangent to the boundary, such integral
curves either lie entirely over $\pa X$ or lie over $\pa X$ only at
their limit points.  Interior and boundary integral curves can meet
only at limit points in $\{x=\bar\eta=0\} \subset \Sbstar X.$

It is helpful in studying the integral curves of $Z$ to introduce the
following way of measuring their lengths: Let $\gamma$ be an integral curve
of $Z$ over $X^\circ$.  Let $k$ denote a Riemannian metric on $\Sbstar
X^\circ$ such that $k(Z,Z)=1$.  Let
\begin{equation}
\omega= x k(\cdot, Z) \in \Omega^1(\Sbstar X).
\label{omega}
\end{equation}
Then
$$
\int_\gamma \omega= \int_\gamma x k(d\gamma/ds, Z)\, ds =\int_\gamma \frac x{\sqrt g}
k(K_g, Z)\, ds=\int_\gamma ds =\text{length}(\gamma) 
$$
where $s$ parametrizes $\gamma$ as an integral curve of $K_g/\sqrt{g}$, the
unit speed geodesic flow.  With this motivation in mind, we now define, for
each $t\in\RR_+$, two relations in $\Sbstar X$, a ``geometric'' and a
``diffractive'' relation.  These correspond to the two different
possibilities for geodesic flow through the boundary.
\begin{definition}\label{defn:relations}
Let $p, q \in \Sbstar X$.  We write
$$
p \gtilde q
$$
if there exists a \emph{continuous,} piecewise smooth curve $\gamma: [0,1] \to
\Sbstar X$ with $\gamma(0)=p$, $\gamma(1)=q$, such that $[0,1]$ can be
decomposed into a finite union of closed subintervals $I_j$, intersecting at
their endpoints, where
\begin{enumerate}
\item on each $I_j^\circ$, $\gamma$ is a
(reparametrized) positively oriented integral curve of $Z$ in $\Sbstar X,$
\item
On successive intervals $I_j$ and $I_{j+1}$, interior and boundary curves
alternate,
\item
$\int_\gamma \omega =t$, with $\omega$ as defined
in \eqref{omega}.
\end{enumerate}

We write
$$
p \dtilde q
$$
if there exists a piecewise smooth (not necessarily continuous) curve
$\gamma: [0,1] \to \Sbstar X$ with $\gamma(0)=p$, $\gamma(1)=q$, such that
$[0,1]$ can be decomposed into a finite union of closed subintervals $I_j$,
intersecting at their endpoints, where
\begin{enumerate}
\item on each $I_j^\circ$, $\gamma$ is a
(reparametrized) positively oriented integral curve of $Z$ in $\Sbstar X^\circ,$
\item
the final point of $\gamma$ on $I_j$ and the initial point of $\gamma$ on
$I_{j+1}$ lie over the same component $Y_\alpha$ of $\pa X,$
\item
$\int_\gamma
\omega =t$.
\end{enumerate}
\end{definition}

Integral curves of $Z$ over $X^\circ$ are lifts of geodesics in $X^\circ,$ and it
follows from \eqref{spray} that the maximally extended integral curves of
$Z$ in $\Sbstar_{\pa X} X$ are lifts of geodesics of length $\pi$ in $\pa
X$ (see \cite{Melrose-Wunsch:cone} for details), hence:
\begin{proposition}
$p\gtilde q$ iff $p$ and $q$ are connected by a (lifted) geometric geodesic of
length $t$.

$p\dtilde q$ iff $p$ and $q$ are connected by a (lifted) diffractive geodesic of
length $t$.
\end{proposition}

A very important feature of these equivalence relations, proved in
\cite{Wunsch:Poisson}, is the following:
\begin{proposition}\emph{(\cite{Wunsch:Poisson}, Prop.~4)}\label{proposition:closed}
The sets $\{(p,q,t): p\gtilde q\}$ and $\{(p,q,t): p\dtilde q\}$ are closed
subsets of $\Sbstar X \times \Sbstar X \times \RR_+$.
\label{prop:closed}
\end{proposition}

\subsection{Propagation and diffraction of singularities}\label{section:MW}
We now recall the key propagation results from
\cite{Melrose-Wunsch:cone}.  In order to do this, we must introduce
the notions of \emph{coisotropic regularity} and \emph{nonfocusing}
with respect to a conic coisotropic submanifold $\coiso$ of $T^*M^\circ.$
We will usually take $\coiso=\flowout_\alpha \sim\flowout_\alpha^S$ with $S$ taken larger
than all times on which we will be considering propagation of
singularities (and hence ignored in our notation).

Let $\module$ denote the space of compactly supported pseudodifferential operators $A \in
\Psi^1(M^\circ)$ that are characteristic on
$\coiso.$  Let $\algebra$ be the filtered algebra generated by
$\module$ over $\Psi^0(M),$
with $\algebra^k=\algebra \cap \Psi^k(M).$ Fix a
Sobolev space $\hilbert.$  Fix a compact set $K \subset S^*M^\circ.$

We now recall the following definitions from \cite{MVW:edges}.
\begin{itemize}
\item\label{coisocond} We say that $u$ has coisotropic regularity with
  respect to $\coiso$ of order $k$ relative to $\hilbert$ in $K$ if
  there exists $A \in \Psi^0(M),$ elliptic on $K$ and supported over
  $M^\circ$ such that $\algebra^k A u \subset \hilbert.$ We say that
  $u$ has coisotropic regularity relative to $\hilbert$ in $K$ if it
  has coisotropic regularity of order $k$ for all $k$>
\item\label{nfcond} We say that $u$ satisfies the nonfocusing
  condition of degree $k$ with respect to $\coiso$ relative to
  $\hilbert$ on $K$ if there exists $A \in \Psi^0(M),$ elliptic on $K$
  and supported over $M^\circ$ such that $A u \subset \algebra^k
  \hilbert.$ We say that $u$ satisfies the nonfocusing condition
  relative to $\hilbert$ on $K$ if it satisfies the condition to some
  degree.
\end{itemize}
If $K$ is omitted (which will only be the case when the distribution
is microsupported away from the boundary), the relevant condition is assumed to
hold on all of $S^*M^\circ$.

In the special case when $\coiso=\flowout_\alpha$ and we work in a
collar neighborhood of the boundary component $Y_{\alpha}$ in which
the metric has the form \eqref{semiproduct}, these definitions
simplify considerably, as the module $\module$ is then generated by
the operators $\pa_{y_j}.$ A distribution $u$ microsupported over such
a neighborhood is coisotropic of order $k$ relative to $\hilbert$ iff
$\ang{\Lap_Y}^{k/2} u \in \hilbert,$ while it satisfies the
nonfocusing condition relative to $\hilbert$ iff there exists $N$ with
$\ang{\Lap_Y}^{-N} u \in \hilbert.$ (Here $\Lap_Y$ denotes the
Laplacian in the $y$-variables with respect to the family of metrics
$h(x)$ on $Y_\alpha.$)

We are now in a position to recall the main results of
\cite{Melrose-Wunsch:cone}.  (Our notation however sticks more closely
to that of \cite{MVW:edges}, which treats the more general case of
edge manifolds.)

Note that here and henceforth we employ the notation $s-0$ to mean $s-\ep$ for every
$\ep>0.$
\begin{proposition}\label{proposition:MW}
Let $u \in \mathcal{C}(\RR; \dom_r)$ be a solution to the wave equation on
$M.$ 
Fix a point $\rho \in \flowout_\alpha$ outgoing with respect to the
cone point $Y_\alpha.$
\begin{enumerate}
\item
If $u$ is microlocally in $H^s$ on all incoming bicharacteristics in
  $\flowout_\alpha$ that are \emph{diffractively related} to $\rho$ then
  $u \in H^s$ microlocally at $\rho.$
\item
Assume that $u$ is nonfocusing with respect to $H^s$ on
$\flowout_\alpha$ and has no wavefront set along
incoming bicharacteristics in $\flowout_\alpha$ that are \emph{geometrically
  related} to $\rho.$  Then $u \in H^{s-0}$ microlocally at $\rho,$
and enjoys coisotropic regularity relative to $H^{s-0}$ in a
neighborhood of $\rho.$
\end{enumerate}
\end{proposition}
\begin{remark} 
  \mbox{}
  \begin{itemize}
  \item
    The coisotropic regularity part of this result is left slightly implicit in the results of
    \cite{Melrose-Wunsch:cone} but follows easily by interpolating
    part (iii) of Theorem~8.1 of that paper, which gives coisotropic
    regularity relative to \emph{some}  Sobolev space, with the overall
    regularity of the solution microlocally near $\rho.$
  \item The following consequence is more germane to what follows (and
    perhaps easier to digest): Fix $s,k$ with $r<k<s.$ Say our
    solution is nonfocusing with respect to $H^s.$ If a singularity in
    $\WF^k$ (the set where $u \notin H^k$ microlocally) arrives at
    $Y_\alpha$ along just one ray in $\flowout_\alpha,$ with $u$
    microlocally smooth along all other arriving rays, then it may
    produce milder singularities, at worst in $H^{s-0}$ and
    coisotropic relative to this space, along rays
    \emph{diffractively} related to the incoming singularity, but has
    strong singularities in $\WF^k$ along (at least some of) those
    \emph{geometrically} related to it.
\end{itemize}
\end{remark}
The crucial example is of course the fundamental solution
$$
u=\frac{\sin t\sqrt{\Lap}}{\sqrt{\Lap}} \delta_p
$$
with $p$ close to a boundary component $Y_\alpha.$
This solution is overall in $H^{-n/2+1-0},$ but it satisfies the
nonfocusing condition relative to $H^{1/2-0}.$  When the 
singularity strikes $Y_\alpha$ is produces strong singularities (no
better than $H^{-n/2+1-0}$) along the geometric continuations of the
geodesic from $p$ to $Y_\alpha$ but only weaker singularities in
$H^{1/2-0}$ along the rest of the points in a spherical wavefront
emanating from $Y_\alpha,$ along which the solution in fact enjoys
\emph{Lagrangian} regularity.

The same reasoning that applies to the fundamental solution in fact
applies to a solution given by a Lagrangian distribution with respect
to a Lagrangian manifold intersecting $\flowout_\alpha$ transversely.
In this paper, however, we will be concerned with a slightly different
setting, in which our hypotheses on the solution are themselves that
of coisotropic regularity, but with respect to a \emph{different}
coisotropic manifold, intersecting $\flowout_\alpha$ transversely (in
particular, with respect to $\flowout_\beta$ for some $\beta$).  In
applying Proposition~\ref{proposition:MW} above, we will need a result
in order to verify the nonfocusing hypotheses.  This result is
discussed in the next section.

\subsection{Intersecting coisotropics}
\label{sec:inters-cois}

In this section we prove the following result allowing us to verify
the nonfocusing hypotheses of Proposition~\ref{proposition:MW}.
\begin{proposition}\label{proposition:nonfocusingholds}
Let $\coiso$ denote a conic coisotropic manifold of codimension $n-1$
intersecting $\flowout_\alpha$ in such a way that the isotropic
foliations of the two coisotropic manifolds are transverse.  Let $u$
enjoy coisotropic regularity with respect to $\coiso,$ relative to
$H^s.$  Then $u$ is nonfocusing with respect to $\flowout_\alpha$
relative to $H^{s+(n-1)/2-0}.$
\end{proposition}

The proof of the proposition relies on two lemmas.  We first show it
in a model setting:
\begin{lemma}\label{lemma:coisoNF}
In $\RR^{n_1+n_2}$ with coordinates $z \in \RR^{n_1},$ $z' \in
\RR^{n_2},$ let $\coiso,$ $\tcoiso$ denote the two model
coisotropic manifolds
$$
\coiso=\{\zeta=0\},\ \tcoiso=\{z=0\}.
$$
If a compactly supported distribution enjoys coisotropic
regularity with respect to $\tcoiso$ relative to $H^s$ then it is nonfocusing with
respect to $\coiso$ relative to $H^{s+n_1/2-0}.$  That is to say,
there exists $N \gg 0$ such that
$$
\ang{\Lap_{z}}^{-N} u \in H^{s+n_1/2-0}_{\loc}(\RR^{n_1+n_2}).
$$
\end{lemma}
Note that $n_1$ is the dimension of the leaves of the isotropic foliation.

\begin{proof}
By applying powers of $\Lap$ we may reduce to the case $s=0.$
Coisotropic regularity with respect to $\tcoiso$ relative to $L^2$
means that
$$
z^{\alpha} D_{z'}^\beta  D_z^\gamma u \in L^2,\ \text{if } \smallabs{\alpha}=\smallabs{\beta}+\smallabs{\gamma},
$$
as the vector fields $z_i D_{z_j}$ and $z_i D_{z'_k}$ have symbols
cutting out $\tcoiso.$

In particular, the iterated regularity under $z_i D_{z_k}$ means that we have
$$
z^\alpha D_{z'}^\gamma u \in L^2(\RR^{n_2}_z;I^{(0)}(\RR^{n_1}; 0)),\  \smallabs{\alpha}=\smallabs{\gamma},
$$
where $I^{(s)}(\RR^{n_1}; 0)$ denotes the space of functions $v \in
H^s(\RR^{n_1})$ enjoying iterated regularity under vector fields $z_i
D_{z_k},$ i.e., conormal regularity with respect to the origin.

Now we claim that if $\rho<n_1/2$ is positive,
$$
\smallabs{z}^{-\rho} I^{(0)}(\RR^{n_1}; 0)\subset H^{-n_1/2-0}(\RR^{n_1}).
$$
Indeed, interpolation shows that for a compactly supported $u \in  I^{(0)}(\RR^{n_1}; 0),$ 
$$
\Lap_z^{s/2} \abs{z}^{s} u \in L^2
$$
for all $s>0.$  Hence Sobolev embedding 
shows that
$I^{(0)}$ is contained in $\smallabs{z}^{-n_{1}/2-0} L^\infty.$ Thus, for
$\rho < n_{1}/2$, $\smallabs{z}^{-\rho}I^{(0)}(\reals^{n_{1}};0)$ is
contained in $\smallabs{z}^{-n_{1}+\epsilon}L^{\infty}$ for some
$\epsilon > 0$.  This in turn implies that for compactly supported $u \in
\smallabs{z}^{-\rho}I^{(0)}(\reals^{n_{1}};0)$, $\check{u}$ is bounded
and so $\langle \zeta\rangle^{-n_{1}/2 - \epsilon} \check{u}\in L^{2}$
for all $\epsilon$.  Taking the inverse Fourier transform yields
$u \in H^{-n_{1}/2-\epsilon}(\reals^{n_{1}})$ for all $\epsilon > 0$.

Thus, applying powers of $z^\alpha D_{z'}^\gamma u$ (and
again interpolating to deal with fractional powers) we find that
$$
u \in H^s(\RR^{n_2}; H^{-n_1/2-0}(\RR^{n_1})),\ s<n_1/2.
$$
Finally, this implies that for $N>n_1/2,$
$$
\ang{\Lap_{z}}^{-N} u \in H^s(\RR^{n_2}; H^{n_1/2}(\RR^{n_1})),\ s<n_1/2.
$$
Since $H^{n_1/2-0} (\RR^{n_2}; H^{n_1/2}(\RR^{n_1})) \subset
H^{n_1/2-0}(\RR^{n_1+n_2})$ this implies nonfocusing with respect to
$n_1/2-0$ as desired.
\end{proof}

Lemma~\ref{lemma:coisoNF} is in fact quite general, as the form of the
coisotropic distributions employed there is a normal form for
intersecting conic coisotropic manifolds with transverse
foliations:
\begin{lemma}\label{lemma:coisonormalform}
Let $\coiso, \tcoiso$ be conic coisotropic submanifolds of
$T^*(\RR^n),$ each of codimension $k<n,$ and intersecting at $\rho \in
T^*(\RR^n)$ in such a way that the isotropic foliation of each is
transverse to the other.  Then there exist local symplectic
coordinates $(z,z',\zeta,\zeta')$ in which $\rho$ lies at the origin, and
$$
\coiso=\{\zeta=0\},\ \tcoiso=\{z=0\}.
$$
\end{lemma}
\begin{proof}
We choose a degree-one homogeneous function $f$ vanishing simply
along $\coiso.$  Thus its Hamilton vector field $\hamvf_f$ has degree zero and is by definition tangent to $\coiso$
and transverse to $\tcoiso.$  On the flowout of $\tcoiso$ by
$\hamvf_f$ we let $g$ denote the homogeneous, degree-zero function measuring the time of
flowout from $\tcoiso;$ we may further extend $g$ to a homogeneous
degree-zero function (also denoted $g$) on $T^*\RR^n$ such that
$$
\{f,g\}=1
$$
and with
$$
g=0,\ dg \neq 0 \text{ on } \tcoiso
$$
(To make this extension, take coordinates, not necessarily symplectic,
in which $\tcoiso$ lies along coordinate axes, and define $g$ locally to be the
time to flow out to a nearby point from a hyperplane containing $\tcoiso$ and transverse
to $\hamvf_f.$)

Now by Theorem~21.1.9 of \cite{Hormander:v3}, we may extend the
coordinates
$$
(z_k,\zeta_k) \equiv (g,f)
$$
to a full system of homogeneous symplectic coordinates.  Thus, we have
locally achieved
$$
\coiso \subset \{\zeta_k=0\},\ \tcoiso \subset \{z_k=0\}.
$$
Since $\hamvf_{\zeta_k},\hamvf_{z_k}$ are respectively tangent to
$\coiso, \tcoiso,$ these manifolds are locally \emph{products} of the
form
$$
\coiso' \times \{(z_k \in \RR, \zeta_k=0)\},\ \tcoiso' \times \{(z_k
=0, \zeta_k\in \RR)\}
$$
with $\coiso', \tcoiso'$ coisotropic submanifolds of $T^*\RR^{n-1}$
satisfying the hypotheses of the lemma with $n$ replaced by $n-1$ and
$k$ replaced by $k-1$.

The result then follows by induction.
\end{proof}

\begin{proof}[Proof of Proposition~\ref{proposition:nonfocusingholds}]
  Proposition~\ref{proposition:nonfocusingholds} now follows from the
  two lemmas above: we can bring the two coisotropic manifolds
  $\flowout_\alpha$ and $\coiso$ into the normal form given by
  Lemma~\ref{lemma:coisonormalform}.  Lemma~\ref{lemma:coisoNF} then
  shows that nonfocusing holds, with $n_1=n-1,$ the dimension of the
  isotropic foliation (which in the case of $\flowout_\alpha$ has
  leaves obtained by varying $y \in Y_\alpha$).
\end{proof}

\subsection{Geometric Assumptions}\label{section:assumptions}
We are now able to state our geometric assumptions:
\begin{assumption}\label{assumption:escape}
  Let $\Omega\supset K$ be an open set with $X\backslash \Omega \sim
  \RR^n \backslash B^n(0,R_1)$ for some $R_1\gg 0.$ We assume that
  there exists a time $T_0>0$ such that any \emph{geometric} geodesic
  starting in $S_K^* X^\circ$ leaves $\Omega$ in time less than
  $T_{0}$ (and does not come back, by convexity of the ball).
\end{assumption}
\begin{assumption}\label{assumption:collinear}
No geometric geodesic passes through three cone points.
\end{assumption}
\begin{assumption}\label{assumption:conjugate}
  No two cone points $Y_\alpha, Y_\beta$ are conjugate to one another
  along geodesics in $S^*X^\circ$ of lengths less than $T_0$ in the
  sense that whenever $s+t\leq T_0,$ $\flowout_\alpha^s$ and
  $\flowout_\beta^t$ intersect transversely for each
  $\alpha,\beta$.
\end{assumption}
\begin{remark}
  Assumption~\ref{assumption:escape} is a quantitative statement of
  non-trapping of geodesics that do not hit the cone points.

  Assumption~\ref{assumption:collinear} is generic at the formal level of
  dimension counting, as a geodesic arriving at $Y_\gamma$ from $Y_\alpha$
  can be geometrically continued from a family of points in $Y_\gamma$ of
  dimension $n-2$ (those points at distance $\pi$ in $Y_\gamma$ from the
  arrival point); on the other hand there is a (generically) discrete set
  of departure points in $Y_\gamma$ for geodesics of bounded length leading
  to $Y_\beta.$ So we have an intersection of set of codimension-one and a
  set of dimension zero dictating the existence of a geometric geodesic
  through these three points.
\end{remark}

The following result (whose proof is contained in
Appendix~\ref{sec:jacobi-proof}) justifies our use of the term
``conjugate'' above.  We let $\mathcal{V}_b(X)$ denote the space of
``b-vector fields,'' meaning those tangent to $\pa X.$
\begin{proposition}
  \label{proposition:jacobi}
  The coisotropic manifolds $\flowout_\alpha^s$ and $\flowout_\beta^t$
  intersect non-transversely if and only if there exist cone points
  $Y_\alpha,Y_\beta,$ a geodesic $\gamma(t)$ with $\gamma(0) \in
  Y_\alpha,$ $\gamma(t') \in Y_\beta$ (with $t'<s+t$) and a normal
  Jacobi field $W$ along $\gamma$ with $W \in \mathcal{V}_b(X).$
\end{proposition}
\begin{remark}
  \label{remark:jacobi}
\mbox{}
\begin{itemize}
\item This result is equivalent to the same statement for the
  manifolds $\flowoutX_\alpha^s,$ $\flowoutX_\alpha^t,$ as the
  incoming and outgoing components of $\flowout_\alpha^s$ are each
  naturally identified with $\RR_t \times \flowout_\alpha^s.$
\item The transverse intersection stated here is equivalent to the
  transverse intersection of the \emph{isotropic fibers} of the coisotropic
  manifolds, as will be seen in the proof below.  In particular, a point of
  intersection automatically entails that the tangent spaces of both
  manifolds contain the tangent to a geodesic and the radial vector field;
  in both manifolds, the space spanned by these two directions is the base
  of the coisotropic foliation.
\item
The usual description of conjugate points involves the vanishing of
the Jacobi field at the endpoints of the geodesic; here, by contrast,
since $W \in \mathcal{V}_b(X),$ we find that the metric length of $W$
vanishes at the endpoints, hence the variation along $W$ should be
regarded as an admissible one for a one-parameter family of geodesics:
varying the endpoint in $Y_\alpha$ should be regarded as varying the
\emph{direction} of departure from the cone point, with the location
of the end fixed at the ``point'' $Y_\alpha$.
\end{itemize}
\end{remark}

Finally, we note that it always suffices to prove the results of this
paper with $\chi$ replaced by $\widetilde{\chi}$ which is $1$ on
$\supp \chi;$ also, by the convexity of the Euclidean ball, the
hypotheses of the theorems are still satisfied if we replace our given
$K$ by a larger compact set. \emph{Therefore, we will assume without
  loss of generality that our manifold is Euclidean on $\supp
  (1-\chi)$ and that $\supp \chi \subset K.$} We also use $\Omega$ to
denote a small neighborhood of the compact set $K$.

\section{Decomposition of the wave propagator}

Let $$L=\min_{\alpha,\beta} d(Y_\alpha,Y_\beta)$$ denote the minimum
distance between cone points.  Fix $$\delta_A\ll 1,\
\delta_\psi<\frac{L}{200}.$$ 
We let $\psi_\alpha$ be cutoff functions,  each equal to
$1$ in a small neighborhood $x<\delta_\psi/4$ of a single cone point
$Y_\alpha$ and supported in $x<\delta_\psi;$ let $\Upsilon\in \CI(X)$ 
equal $1$ outside $\Omega$ and vanish on $\supp \chi;$
and finally let $A_j$ ($j=1,\dots, N$) be a
pseudodifferential partition of $\Id-\sum\psi_\alpha-\Upsilon$  in which each element is microsupported on a
set of diameter less than $\delta_A$ (with respect to our arbitrary
but fixed metric on $S^*X$) and in particular is microsupported over $K.$
We can then further arrange that
$$
\sum_{j=1}^N A_j+\sum_\alpha \psi_\alpha +\Upsilon-\Id \in \residual.
$$
(The error term in fact can be taken to lie in $\Psi^{-\infty}_c(X^\circ),$
which is to say it is a smoothing operator with Schwartz kernel compactly
supported in $X^\circ,$ and thus in particular vanishing in a neighborhood
of all boundary components.)

We will consider sequences of wave propagators sandwiched between
various of the $A_j$'s: with $J=(j_0,\dots,j_{k+1}),$ set
$$
T_J=A_{j_0} U(t_0) A_{j_1} U(t_1)\dots A_{j_k} U(t_k) A_{j_{k+1}}.
$$
Associated to each such propagator is a \emph{word} $j_0j_1\dots
j_{k+1}$ referring to a sequence of elements of the cover, as well as the
additional data of a time $t_\ell$ associate to each pair of successive
letters $j_{\ell}j_{\ell+1}.$  We say that such a word is \emph{geometrically realizable}
(``GR'') if for each pair of successive sets $j_\ell j_{\ell+1},$ in
the word, there exist $p_\ell \in \WF' A_{j_\ell}$ and $p_{\ell+1} \in
\WF' A_{j_{\ell+1}}$ with
$$
 p_{\ell+1} \gtilde[t_\ell] p_\ell.
$$
We call the word \emph{diffractively realizable} (``DR'') if instead
there are points with
$$
p_{\ell+1} \dtilde[t_\ell] p_\ell.
$$
Note that any GR word is DR, and also that the words are read from
right to left in order to conform with the composition of operators.
We also describe individual pairs of successive letters in a word as
\emph{diffractive} or \emph{geometric interactions} depending on
whether these two-letter subwords are GR or DR.  We say that such a
successive pair of letters \emph{interacts with a cone point} if there
exists a diffractive geodesic of length $t_\ell$ between $p_{\ell} \in
\WF' A_{j_\ell}$ and $p_{\ell+1} \in \WF' A_{j_{\ell+1}}$ that passes
through some boundary component $Y_\alpha$

The following result comes direction from (``diffractive'')
propagation of singularities, Theorem~I.2 of
\cite{Melrose-Wunsch:cone}; its content is simply that singularities
of solutions to the wave equation propagate along diffractive geodesics:
\begin{proposition}
If the word $J$ is not DR then $T_J\in \residual.$
\end{proposition}

\begin{lemma}
  If $\delta_A$ is chosen small enough and $t_0+\dots +t_{k}>2T_0,$
  the word $J=(j_0\dots j_{k+1})$ cannot be GR.
\end{lemma}
\begin{proof} Shrinking $\delta_A=2^{-m}\downarrow 0,$ if there are GR words with
  fixed time intervals, in the limit there must be a geometric
  geodesic of length $2T_0$ that remains in $\Omega,$ which is ruled
  out by assumption.
\end{proof}

We also have the following, more granular, reformulation of
Assumption~\ref{assumption:collinear}, which we will need in what follows:
\begin{lemma}\label{lemma:collinear}
If $\delta_A$ is sufficiently small, there do not exists words of the form $ijk\ell,$
where $jk$ is GR and $ij,$ $jk,$ and $k\ell$ all interact with cone points.
\end{lemma}
\begin{proof}
  If the result fails, then taking $\delta_A=2^{-m}$ gives a family of
  broken geodesics given by the concatenation of $\gamma_0^m,$
  $\gamma_1^m,$ $\gamma_2^m$ where \begin{itemize} \item $\gamma_0^m$ starts
    and $\gamma_2^m$ ends at a cone point, \item The end of $\gamma_i^m$
    and the start of $\gamma_{i+1}^m$ are within distance $2^{-m}$ of
    each other in $S^*X^\circ,$ \item $\gamma_1^m$ undergoes geometric
    interaction with a cone point.\end{itemize} Then taking $m \to
  \infty$ would yield (by compactness of $\pa X$ and Proposition~\ref{proposition:closed}) a
  limiting geodesic passing through three cone points, interacting
  geometrically with the one in the middle; this contradicts
  Assumption~\ref{assumption:collinear}.
\end{proof}

The crucial ingredient in the proof of Theorem~\ref{theorem:Huygens}
is the following lemma which shows that the propagator eventually
locally smooths data microsupported in $\WF'A_m$ by $(n-1)/2-0$
derivatives.  We cannot quite just add together all these terms for
differing $m$ and conclude the theorem, as that still leaves
singularities starting in $\supp \psi_\alpha$ to be dealt with;
however, we will see later on that the singularities near the cone
points in these latter terms can be moved away from the cone points by
applying the propagator for short time.

In order to make iterative arguments, it is convenient to keep track
explicitly of singularities that leave $\Omega$ never to return: we
let $\outgoing$ (for ``outgoing'') denote the subset 
$$
\outgoing=\big\{(t,z,\tau,\zeta) \in T^*(M\backslash \Omega): \ang{z,\zeta/\tau}>0\big\}.
$$
(Here we have abused notation by identifying $X\backslash \Omega$ with
a subset of Euclidean space to which it is isometric.)  The set
$\outgoing$ is mapped to itself by positive time geodesic flow;
moreover, any bicharacteristic starting in $\supp \chi$ that escapes
$\Omega$ lies in $\outgoing$ over $X\backslash \Omega.$  We let
$L^2\, H^s(\outgoing)$ denote the space of distributions that are
microsupported in $\outgoing$ (hence in $\dom_\infty$ on $K$) and lie
in $L^2\, H^s$ (where as before we use the notation $L^2\, H^s$ for
$L^2([0, \widetilde{T}]; H^s)$).  Let $L^2 \,
\energy_s(\outgoing)=L^2\, (H^s(\outgoing)\oplus H^{s-1}(\outgoing)).$

\begin{lemma}\label{lemma:awayfromcone}
There exist $\delta_A,\delta_\psi$ sufficiently small that for each
$m$ and $s,$ for all $t>5T_0,$
$$
U(t) A_m: \energy_s \to L^2 \, \energy_{s+(n-1)/2-0}+L^2\, \energy_s(\outgoing).
$$
\end{lemma}
Since distributions in $L^2\, \energy_s(\outgoing)$ are smooth on $\Omega$
this implies in particular that $\chi U(t) A_m: \energy_s \to \energy_{s+(n-1)/2-0}.$
\begin{proof}
  If all diffractive bicharacteristics starting in $A_m$ escape to
  $\outgoing$ in time less than $5T_0,$ then the result holds, by
  propagation of singularities.  If not, some bicharacteristic must
  hit a cone point within time $T_0$---none can do so in longer time,
  as otherwise it would be a (trivially) geometric bicharacteristic
  remaining in $\Omega$ for time $T_0,$ contradicting
  Assumption~\ref{assumption:escape}.  We now let $s_0$ be the time
  at which the first cone point is reached under the flow.  In time
  $s_0+3\delta_\psi,$ then this particular bicharacteristic is at
  least distance $2\delta_\psi$ from the boundary; hence if $\delta_A$
  is small enough then taking $t_0 =s_0+3\delta_\psi$ we find that any
  singularity starting within distance $\delta_A$ of this one is
  propagated by $U(t_0)$ to lie at distance greater than $\delta_\psi$
  from the boundary, by propagation of singularities.  (Note that we
  choose $\delta_A$ small enough that bicharacteristics of length
  close to $t_0$ starting from $\WF'A_m$ interact with only a single
  cone point.)  Thus, \emph{either} $U(5T_0) A_m$ has range in
  $\energy_s(\outgoing)$ \emph{or} there exists $t_0<T_0$ such that:
$$
U(t_0)A_m=\sum_\ell A_\ell U(t_0) A_m \bmod \residual.
$$
Now some of the words $\ell m$ are DR in time $t_0$ and others (most)
are not.  Those that are not give smoothing terms, so we discard
them.  For those that are, we repeat the construction above, twice,
starting in $\WF A_\ell$ and $\WF A_k$ successively instead of in
$\WF' A_m.$  We may thus write $U(5 T_0) A_m$ as a sum of terms
$$
U(5 T_0) A_m =\sum U(5 T_0-t_0-t_1-t_2) A_j U(t_2) A_k U(t_1) A_\ell U(t_0) A_m+E+R
$$
where all words $jk\ell m$ are DR, and where
$$
E:\energy_{s}\to \energy_s(\outgoing),\ R \in \residual.
$$  The choices of $t_1$ and
$t_2$ in the sum depend on $k$ and $\ell$ just as our choice of $t_0$
depended on how long it took a bicharacteristic in starting in $\WF'
A_m$ to hit a cone point.  Since this dependence is not relevant in
what follows, however, we suppress it in the notation.  Note that each
$t_i$ is less than $T_0,$ as otherwise the bicharacteristics must have
escaped to $\outgoing$ rather than interacting with another cone point.

Now in each word $jk\ell m$ associated with an element of the sum
there are three interactions with cone points.  Some of these
interactions are GR, and some are not.  We encode this by associating
a string of G's and D's to a word, so for instance a diffractive
interaction followed by two geometric gives (reading right to left)
the string $GGD.$ We break our sum into pieces based on this
classification.

For a word containing two successive $D$'s, i.e.\ $GDD,$ $DDG$, or
$DDD,$ we claim that the propagator maps $\energy_s \to
\energy_{s+(n-1)/2-0}.$ The proof is as follows: given initial data in
$\energy_s,$ Proposition~\ref{proposition:MW} tells us that the
first diffractive interaction
results in a solution $U(t) A_m$ that is coisotropic
with respect to $\flowout_\alpha$ (for the relevant cone point
$\alpha$,) relative to $H^{s-0}.$ This distribution then propagates in $T^*M^\circ$
so as to preserve this coisotropic regularity (see Proposition~12.2 of
\cite{Melrose-Wunsch:cone} or Lemma~4.7 of \cite{MVW:edges} from which it
also follows).  Now when the singularities arrive at the next
cone point (say, $Y_\beta$) for the second diffraction, by
Assumption~\ref{assumption:conjugate} (i.e., nonconjugacy) we may apply
Proposition~\ref{proposition:nonfocusingholds} to conclude that the
solution is nonfocusing with respect to $\flowout_\beta$ relative to $H^{s+(n-1)/2-0}.$
Then the
second diffractive interaction puts it in $H^{s+(n-1)/2-0},$ by a second
application of Proposition~\ref{proposition:MW}.

By contrast a word containing $G$ in the middle, i.e., $DGD,$ $GGD,$ $DGG,$ or
$GGG$ cannot be realizable by Lemma~\ref{lemma:collinear}.

This leaves words of the form $GDG$ as the only remaining summands to
treat.  A geometric geodesic passing starting in $\WF'A_j,$ traveling
for time $t_2,$ reaching $\WF' A_i$, and then traveling for time $5
T_0-t_0-t_1-t_2$ must reach $\outgoing$ since the only other option is
for it to reach another cone point, which again would contradict Lemma~\ref{lemma:collinear}.  Thus $U(5 T_0-t_0-t_1-t_2) A_j U(t_2) A_k$
must map singularities to $\outgoing.$

Thus, we have established that every term in the sum representing
$U(5T_0) A_m$ either maps singularities to $\outgoing$ or smooths
them by $(n-1)/2-0$ derivatives.
\end{proof}

\section{Weak non-trapping of singularities}
In this section we prove Theorem~\ref{theorem:Huygens}, which tells us that
weak non-trapping of singularities holds.

The proof consists of two steps.  To start, we will prove the theorem when $s=(n-1)/2-0.$  To accomplish this, we apply 
Lemma~\ref{lemma:awayfromcone} as follows.  We decompose
\begin{equation}\label{decomposition}
\chi U(t) \chi = \sum_j \chi U(t) \chi A_j + \sum_\alpha \chi U(t) \chi \psi_\alpha.
\end{equation}
Then for $t>5T_0$ the first sum has the desired mapping property by
Lemma~\ref{lemma:awayfromcone}
since
$\WF' \chi A_j \subset \WF' A_j$.  We
deal with the second as follows: pick $\tau>2\delta_\psi,$ and smaller
than $L/50$ (recall that $L$ is the minimum distance between
cone points).
By propagation of singularities, if $\WF v \subset \WF' \psi_\alpha$ then 
$$
\psi_\beta U(\tau) v \in \energy_\infty \text{ for all }\beta,
$$
hence
$$
U(\tau) v-\sum_{j} A_j U(\tau) v \in L^2\, \energy_\infty+L^2 \energy_s(\outgoing).
$$
Thus, we may rewrite the second sum in \eqref{decomposition} as
$$
\sum_{\alpha} \sum_j \chi U(t-\tau)\big( A_j U(\tau)\chi
\psi_\alpha\big) \bmod \residual,
$$
and again applying Lemma~\ref{lemma:awayfromcone}, this time with
$t>\tau+5 T_0,$ shows that these terms, too, enjoy the desired mapping
properties.  This concludes the first step in the proof.

To finish the proof, we need to show that further smoothing occurs as
time evolves.  To this end, note that we may iterate the result
obtained above as follows.  Given $f=(f_0,f_1) \in \energy_s$ we 
choose $\chi_1 \in \CcI(X)$ equal to $1$ on $K$ and
split
$U(5T_0) \chi f=\chi_1 U(5T_0)
\chi f+(1-\chi_1) U(5T_0) \chi f$ so that for $t>5T_0$
$$
U(t)\chi f = u+v
$$
where
$$
u = U(t-5T_0) \chi_1 U(5 T_0)\chi f,\ v = U(t-5T_0) (1-\chi_1) U(5 T_0)\chi f.
$$
Thus by our previously established results, if $\chi_1$ is chosen with
support sufficiently close to $K$, $u \in L^2([5T_0, \widetilde{T}];
\energy_{s+(n-1)/2-0})$ and $v \in L^2([5T_0, \widetilde{T}];
\energy_{s}(\outgoing)).$   For $t>10 T_0$ we now employ the
smoothing result established above with $$u(5 T_0) =\chi_1 U(5T_0)\chi f\in \energy_{s+(n-1)/2-0}$$ now
functioning as our initial data (and the previous localizer $\chi$ replaced by $\chi_1$) to obtain
\begin{equation}
\begin{aligned}
u(t) &=U(t-5T_0) u(5T_0)\\
&\in L^2([10T_0, \widetilde{T}]; \energy_{s+2(n-1)/2-0})+ L^2([10T_0, \widetilde{T}]; \energy_{s+(n-1)/2-0}(\outgoing)).
\end{aligned}
\end{equation}
Hence overall $$U(t)\chi f \in  L^2([10T_0, \widetilde{T}]; \energy_{s+2(n-1)/2-0})+ L^2([10T_0, \widetilde{T}]; \energy_{s}(\outgoing)).$$
Further iteration of this argument now yields smoothing by $k (n-1)/2-0$ derivatives
after time $5k T_0.$\qed

\section{Exterior polygonal domains}
\label{sec:exter-polyg-doma}

In this section we show that the very weak Huygens' principle of
Theorem~\ref{theorem:Huygens} also holds for the wave equation exterior to
a non-trapping polygonal obstacle with Dirichlet or Neumann boundary
conditions.  In particular, we suppose that $\Omega\subset \reals^{2}$ is a
compact region with piecewise linear boundary.  We further suppose that the
complement $\reals^{2}\setminus \Omega$ is connected, that no three
vertices of $\overline{\Omega}$ are collinear, and that
$\reals^{2}\setminus \Omega$ is non-trapping in the sense that the doubling
described next satisfies Assumption~\ref{assumption:escape}.  This
assumption is generically equivalent to the requirement that all billiard
trajectories missing the vertices escape to infinity in some uniform time.

We now form a manifold $X$ by gluing two copies of
$\reals^{2}\setminus \Omega$ along their boundaries.  This process
yields a Euclidean surface $(X,g)$ with conic singularities satisfying
the assumptions of Section~\ref{section:assumptions} (but with two
Euclidean ends).  As its proof goes through verbatim for a manifold
with two Euclidean ends rather than one, Theorem~\ref{theorem:Huygens}
then holds for $(X,g)$.

Suppose now that $\Lap$ is the Dirichlet or Neumann extension of the
Laplacian on $\reals^{2}\setminus \Omega$.\footnote{By an observation of
Blair, Ford, Herr, and Marzuola~\cite{BFHM}, both the Dirichlet and
Neumann Laplacians are taken to the Friedrichs extension of the
Laplacian on the conic doubled manifold.}  The method of images then
shows that Theorem~\ref{theorem:Huygens} holds for
$\reals^{2}\setminus \Omega$.  Indeed, by solving the wave equation on
the double $X$ and then summing (respectively, taking the difference)
over the two copies one obtains a solution for the wave equation with
the Neumann (respectively, Dirichlet) extension of the Laplacian on
$\reals^{2}\setminus \Omega$.

\section{From weak non-trapping to exponential decay}\label{section:Vainberg}

In this section, we recapitulate the argument of Vainberg
\cite{Vainberg} as repackaged by Tang--Zworski~\cite{TZ} in the
setting of \emph{weak} non-trapping of singularities in order to
deduce our resolvent estimate Theorem~\ref{theorem:resest} and hence
exponential energy decay for the wave equation in odd dimensions and
the resonance wave expansion (Corollary~\ref{corollary:resexp}).

Having established that the weak non-trapping of singularities holds both in
the settings of manifolds with cone points and of exterior domains to polygons, it will behoove us to
adopt a formalism for passing from this property to resolvent estimates
that will simultaneously apply in both cases.  For this reason we now adopt
the ``black-box'' formalism as used in \cite{TZ}.

\subsection{Preliminaries}
\label{sec:preliminaries}

We start by recalling the framework of ``black box'' scattering from
Sj{\"o}strand--Zworski~\cite{SZ}, as used by Tang--Zworski~\cite{TZ}.  The
following presentation follows \cite{TZ} closely.

We consider a complex Hilbert space
$\mathcal{H}$ with an orthogonal decomposition
\begin{equation*}
  \mathcal{H} = \mathcal{H}_{R_{0}} \oplus L^{2}\left(\reals^{n}\setminus B(0,R_{0})\right),
\end{equation*}
where $R_{0}>0$ is fixed.  We assume that $P$ is a self-adjoint
operator, $P: \mathcal{H} \to \mathcal{H}$, with domain $\mathcal{D}
\subset \mathcal{H}$, satisfying the following conditions:
\begin{align*}
  \mathbbm{1}_{\reals^{n}\setminus B(0, R_{0})}\mathcal{D} &=
  H^{2}(\reals^{n}\setminus B(0, R_{0})), \\
  \mathbbm{1}_{\reals^{n}\setminus B(0,R_{0})}P &= \Lap
  |_{\reals^{n}\setminus B(0,R_{0})} \\
  (P + i )^{-1} &\text{ is compact} \\
  P \geq -C, &\quad C \geq 0.
\end{align*}

Under the above conditions, it is known that the resolvent $R(\lambda)
= (P-\lambda^{2})^{-1}: \mathcal{H}\to \mathcal{D}$ meromorphically
continues from $\{ \lambda : \Im \lambda > 0, \lambda^{2} \notin
\sigma(P)\}$ to the whole complex plane $\complexes$ when $n$ is odd
or to the logarithmic plane $\Lambda$ when $n$ is even, as an operator
from $\mathcal{H}_{\text{comp}}$ to $\mathcal{D}_{\text{loc}}$ with
poles of finite rank.  We denote by $\mathcal{D}_{s}$ the spaces given
by $(P+i)^{-s/2}\hilbert$.

The poles (i.e., the resonances of $P$) wil be denoted $\Res (P);$ we count
them with multiplicity, denoted $m(\lambda)$.  We require an additional condition on $P$
to guarantee a polynomial bound on the resonance counting function.  

To formulate the additional condition, we use $P$ to construct a
self-adjoint reference operator $P^{\#}$ on
\begin{equation*}
  \mathcal{H}^{\#} = \mathcal{H}_{R_{0}} \oplus L^{2}\left( M
    \setminus B(0,R_{0})\right)
\end{equation*}
as in~\cite{SZ} by gluing the ``black box'' into a large torus instead of
Euclidean space; here $M = \left( \reals / R\integers\right)^{n}$ for some $R >
R_{0}$.  Let $N(P,I)$ denote the number of eigenvalues of $P^{\#}$ in
the interval $I$.  The assumption we need is then
\begin{equation}\label{blackboxassumption}
  N(P^{\#}, [-C,\lambda]) = O(\lambda^{n^{\#}/2}), \quad \lambda \geq 1
\end{equation}
for some number $n^{\#}\geq n$.\footnote{This implies (see~\cite{SZ} and the
references of~\cite{TZ}) that the resonance counting function $N(r)$
satisfies
\begin{equation*}
  N(r) = \sum_{\substack{\lambda \in \Res (P) \\ |\lambda| \leq r, \arg \lambda
    < \theta}} m(\lambda) \leq C_{\theta}r^{n^{\#}}.
\end{equation*}}

In addition we have a polynomial bound for the logarithm of the norm
of the cutoff resolvent in the whole complex plane away from the
resonance set in odd dimensions and in neighborhoods of the real axis
in even dimensions: for any $\chi \in C^{\infty}_{c}(\reals^{n})$,
$\chi =1$ near $B(0,R_{0})$,
\begin{equation*}
  \norm[\mathcal{H}\to\mathcal{H}]{\chi R(\lambda)\chi} \leq
  C_{\theta}e^{C_{\theta}|\lambda|^{n^{\#}+\epsilon}}, \quad \lambda
  \in \{ \lambda : \arg \lambda < \theta\} \setminus
  \bigcup_{\lambda_{j}\in \Res (p)} B(\lambda_{j}, \langle
  \lambda_{j}\rangle ^{-n^{\#}-\epsilon}).
\end{equation*}
Note that the condition \eqref{blackboxassumption} is satisfied for
the polygonal exteriors of Section~\ref{sec:exter-polyg-doma} as well
as for conic manifolds, with $\mathcal{H}_{R_0}$ taken to be $L^2(K)$
in either case, and $\mathcal{D}$ the domain of the square root of the
appropriate Laplace operator (i.e., with boundary conditions in the
polygonal case and simply the Friedrichs extension in the conic case).

The wave group of a black box perturbation can be defined abstractly
as in Proposition~2.1 of Sj{\"o}strand--Zworski~\cite{SZ2}:
\begin{equation*}
  U(t) = \exp  i t
    \begin{pmatrix}
      0 & I \\ P & 0
    \end{pmatrix}.
\end{equation*}
The entries of the matrix representation of $U(t)$ are
\begin{equation*}
  U(t) =
i  \begin{pmatrix}
    D_{t}\mathcal{U}(t) & \mathcal{U}(t) \\ D_{t}^{2}\mathcal{U}(t) & D_{t}\mathcal{U}(t)
  \end{pmatrix},
\end{equation*}
where the strongly continuous family of operators $\mathcal{U}(t):
\mathcal{D}_{s} \to \mathcal{D}_{s+1}$ can be identified as the
solution operator of the following initial value problem:
\begin{align*}
  \left( D_{t}^{2} - P\right) \mathcal{U}(t)g &= 0 \text{ for
  }t\in\reals \\
  \mathcal{U}(0)g &= 0 \\
  \pa_t \mathcal{U}(0)g &= g
\end{align*}
where $g\in \mathcal{H}$.  In other words,
\begin{equation*}
  \mathcal{U}(t) =  \frac{\sin t\sqrt{P}}{\sqrt{P}}.
\end{equation*}
Since $\mathbbm{1}_{\reals^{n}\setminus
  B(0,R_{0})}\mathcal{U}(t)\mathbbm{1}_{\reals^{n}\setminus
  B(0,R_{0})}$ maps Schwartz functions to tempered distributions we
can describe it by its Schwartz kernel $\mathcal{U}(t,x,y)$, which is
a distribution in on $\reals \times (\reals^{n}\setminus
B(0,R_{0}))\times (\reals^{n}\setminus B(0,R_{0}))$.

\subsection{The resolvent estimate}
\label{sec:resolvent-estimate}

We work in the framework of black-box scattering described in
Section~\ref{sec:preliminaries}.  As above, we use the notation $\Ut$
to represent the sine wave propagator $\sin \left( t
  \sqrt{P}\right)/\sqrt{P}$.  In what follows we use
$\mathcal{F}^{-1}$ to denote the inverse Fourier transform:
\begin{equation*}
  \mathcal{F}^{-1}_{t\to\lambda} f (\lambda) = \int _{\reals}e^{it\lambda}f(t)\,dt
\end{equation*}
We use the usual notation in which $\check{f}$ denotes
$\mathcal{F}^{-1}f$.

The main result of this section is an adaptation of an argument of
Vainberg~\cite{Vainberg} as repackaged by Tang--Zworski~\cite{TZ}.  It
states that the very weak Huygens' principle of
Theorem~\ref{theorem:Huygens} implies the resolvent bounds of
Theorem~\ref{theorem:resest}.
\begin{proposition}
  \label{prop:vainberg}
  Suppose that $P$ is a black-perturbation for which the  very weak
  Huygens' principle of Theorem~\ref{theorem:Huygens} holds.  Then
  there exists a $\delta > 0$ such that the cut-off resolvent
  \begin{equation*}
    \chi \left( P - \lambda^{2}\right)^{-1} \chi
  \end{equation*}
  can be analytically continued from $\Im \lambda > 0$ to the region
  \begin{equation*}
    \Im \lambda > - \delta \log \Re \lambda , \quad \Re \lambda >
    \delta ^{-1}
  \end{equation*}
  and for some $C, T> 0$ enjoys the estimate
  \begin{equation*}
    \norm[\mathcal{H}\to \mathcal{H}]{\chi \left(P - \lambda^{2}\right)^{-1}\chi}
    \leq C |\lambda|^{-1} e^{T|\Im \lambda|}
  \end{equation*}
  in this region.
\end{proposition}

\begin{lemma}
  \label{lem:fourier-transform}
  Suppose that $H_{1}$ and $H_{2}$ are Hilbert spaces, and that
  $N(t):H_{1}\to H_{2}$ is a family of bounded operators that have $k$
  continuous derivatives in $t$ when $t\in\reals$, depend analytically
  on $t$ when $\Re t > T > 0$, and are equal to zero when $t < 0$.
  Suppose that there are constants $j_{0}$, $k \geq j_{0}+2$, and
  $C_{j}$ so that for all $0 \leq j \leq k$,
  \begin{equation*}
    \norm{\frac{\partial^{j}}{\partial t^{j}}N(t)} \leq C_{j}
    |t|^{j_{0}-j}, \quad \text{for }\Re t > T.
  \end{equation*}
  Then the operator
  \begin{equation*}
    \check{N}(\lambda) = \mathcal{F}^{-1}_{t\to \lambda}N(t): H_{1}\to H_{2},
    \quad\text{for } \Im \lambda > 0
  \end{equation*}
  can be continued analytically to the domain $- \frac{3\pi}{2} <
  \operatorname{arg}\lambda < \frac{\pi}{2}$ and when $|\lambda | \geq
  1$, it satisfies the estimate
  \begin{equation*}
    \norm{\check{N}} \leq C_{j}|\lambda|^{-j}e^{T|\Im\lambda|}, \quad
    \text{for }j = 0 , \ldots , k.
  \end{equation*}
\end{lemma}

\begin{proof}[Proof of Lemma~\ref{lem:fourier-transform}]
  This is a straightforward adaptation of the proof due to
  Vainberg~\cite[Lemma 4 on page 346]{Vainberg} (see also Lemma 3.1
of \cite{TZ}).  We include it here
  for completeness. 

  The operator $\check{N}$ is defined and depends analytically on
  $\lambda$ when $\Im \lambda \geq 0$ and, for $0 \leq j \leq k$,
  \begin{equation*}
    \check{N} = (-i\lambda)^{-j}\int_{0}^{\infty}e^{i\lambda
      t}\frac{\partial^{j}}{\partial t^{j}}N(t) \, dt .
  \end{equation*}
  Let $\Gamma_{\pm}$ denote the contours in the complex $t$-plane
  formed by the interval $[0,T]$ and the rays $[T, T\pm i \infty)$.
  When $j \geq j_{0}+2$, we have
  \begin{align*}
    \check{N} &= (-i\lambda)^{-j} \int_{\Gamma_{+}}e^{i\lambda t}
    \frac{\partial^{j}}{\partial t^{j}}N(t) \, dt \quad 0 <
    \operatorname{arg} \lambda < \frac{\pi}{2}, \\
    \check{N} &= (-i\lambda)^{-j}\int_{\Gamma_{-}}e^{i\lambda
      t}\frac{\partial^{j}}{\partial t^{j}}N(t) \, dt \quad
    \frac{\pi}{2} < \operatorname{arg}\lambda < \pi .
  \end{align*}
  The first formula allows us to continue $\check{N}$ analytically to
  the half-plane $\Re \lambda \geq 0$ and the second formula to the
  half-plane $\Re \lambda < 0$.  This also yields the estimates for $j
  \geq j_{0}+2$ and therefore for all $0 \leq j \leq k$.
\end{proof}

\begin{proof}[Proof of Theorem~\ref{theorem:resest}]
  We start by fixing $\chi \in C^{\infty}_{c}(X)$ and $s$ large ($s
  \geq \frac{1-n}{2}+2$ should suffice) and then fixing $T_{0}= T_{s}$
  as in the statement of Theorem~\ref{theorem:Huygens} so that for all
  $t > T_{0}$,
  \begin{equation*}
    \chi U(t) \chi : \energy_{r} \to \energy _{r+s}
  \end{equation*}
  for all $r$.

  In addition to the cutoff function $\chi_{1}\equiv \chi,$ we
  introduce two other spatial cutoff functions, $\chi_{2}$ and
  $\chi_{3}$, so that $\chi_{i}\in C^{\infty}_{c}(X)$,  $P(1-\chi_i) =
  \Lap_0 (1-\chi_i),$
 $\chi_{1}\chi_{2} = \chi_{2}$, and $\chi
  _{2}\chi_{3} = \chi_{3}$.

  We introduce a spacetime cutoff function $\zeta$ so that $\zeta$ is
  independent of the spatial variables $z$ on $K$, $0 \leq \zeta \leq
  1$, and
  \begin{equation*}
    \zeta (t,z) =
    \begin{cases}
      1 & t \leq |z| + T_{0} \\
      0 & t \geq |z|+T_{0}'
    \end{cases}
  \end{equation*}
  for some $T_{0}' \geq T_{0}$.  Finite speed of propagation and our
  weak non-trapping hypothesis imply that
  \begin{equation*}
    (1 - \zeta) \Ut \chi : L^{2} \to \dom_{s}
  \end{equation*}
  for all $t$.

  For any $g \in L^{2} = \dom_{0}$, consider $\zeta \Ut \chi_{1} g$,
  which trivially satisfies
  \begin{align*}
    \left( \PD[t]^{2} - P\right) \zeta \Ut \chi g
    &= - \left( \PD[t]^{2} - P \right) (1-\zeta) \Ut \chi g ,\\
    \Ut[0]\chi g &= 0 ,\\
    \PD[t]\Ut[0]\chi g &= \chi g .
  \end{align*}
  We now define $F(t) g = - \left( \PD[t]^{2} - P\right) (1-\zeta)
  \Ut \chi g$.  Our weak non-trapping assumption implies that
  \begin{equation*}
    F(t)g \in C^{0}\left( \reals_{t} ; \dom _{s-2}\right) \cap C^{s-2}
    \left( \reals_{t}; \dom_{0}\right).
  \end{equation*}
  Note that $F(t) g$ vanishes identically for $t < T_{0}$ and has
  compact support in $t$ for each fixed $z$ (though the size of the
  support depends on $z$).

  We now define $\tilde{R}(\lambda)$ by
  \begin{equation*}
    \tilde{R}(\lambda) = - i \mathcal{F}^{-1}_{t\to \lambda} \left( \zeta
      H(t)\Ut \chi\right),
  \end{equation*}
  where $H(t)$ is the Heaviside function.  A simple calculation shows
  that
  \begin{equation*}
    \PD[t]^{2}\left[ H(t)\Ut \chi g\right] = H(t) \PD[t]^{2}\Ut \chi g -
    i \delta(t) \chi g,
  \end{equation*}
  and so $\tilde{R}(\lambda)$ satisfies
  \begin{align*}
    \left( P - \lambda^{2}\right) \tilde{R}(\lambda) g &= i
    \mathcal{F}^{-1}\left( \left( \PD[t]^{2} - P\right) \zeta H(t) \Ut
      \chi g\right)\\
    &= i \mathcal{F}^{-1}\left( H(t) \left( \PD[t]^{2} - P\right)\Ut
      \chi g\right) + \chi g \\
    &= i \mathcal{F}^{-1}\left( F(t) g\right)(\lambda) + \chi g,
  \end{align*}
  where the last equality holds because the support of $F(t)$ is
  contained in $t \geq 0$.

  We now write $F(t) g = \chi_{2} F(t)g + (1-\chi_{2})F(t) g$ and
  solve an inhomogeneous wave equation on a flat background.  In
  particular, if $\Lap_{0}$ denotes the (flat) Laplacian on
  $\reals^{n}$, we find $V(t)g$ so that it solves
  \begin{align*}
    \left( \PD[t]^{2}- \Lap_{0}\right) V(t) g &= ( 1- \chi_{2})F(t)g \\
    V(0) g = \PD[t]V(0)g &= 0.
  \end{align*}
  Using the cutoff function $\chi_{3}$ (and the fact that
  $\chi_{3}(1-\chi_{2}) = 0$), we observe that
  \begin{align*}
    \left( 1- \chi_{2}\right) F(t) g &= \left( \PD[t]^{2} -
      \Lap_{0}\right)\left( \chi _{3}V(t)g\right) + \left( \PD[t]^{2}
      -
      \Lap_{0}\right) \left( (1-\chi_{3})V(t)g\right) \\
    &= -\left[ \Lap_{0}, \chi_{3}\right] V(t) g + \left( \PD[t]^{2} -
      \Lap_{0}\right)\left( (1-\chi_{3})V(t)g\right).
  \end{align*}

  We now define $R^{\#}(\lambda)$ as follows:
  \begin{equation*}
    R^{\#}(\lambda) = \tilde{R}(\lambda) + i \mathcal{F}^{-1}_{t\to \lambda}
    \left( (1-\chi_{3})V(t)\right)
  \end{equation*}
  Since $P = \Lap_{0}$ on the support of $1-\chi_{3}$, we
  observe that
  \begin{align*}
    \left( P - \lambda^{2}\right) R^{\#}(\lambda) g &= \chi g + i
    \mathcal{F}^{-1}\left( F(t)g - (\PD[t]^{2} - P) (1-\chi_{3})V(t)g
    \right) \\
    &= \chi g + i \mathcal{F}^{-1}\left( F(t)g -
      (1-\chi_{3})(1-\chi_{2})F(t)g - \left[
        \Lap_{0},\chi_{3}\right]V(t)g\right) \\
    &= \chi g + i \mathcal{F}^{-1}\left( \chi_{2}F(t)g - \left[ \Lap_{0},
        \chi_{3}\right]V(t)g\right) \\
    &= \chi \bigg( \id + i \mathcal{F}^{-1}\left( \chi_{2}F(t) - \left[
        \Lap_{0}, \chi_{3}\right]V(t)\right)\bigg)g .
  \end{align*}

  In other words,
  \begin{equation}
    \label{eq:resolvent-identity}
    R^{\#}(\lambda)  = R(\lambda) \chi \bigg( \id + i \mathcal{F}^{-1} \left(
      \chi_{2}F(t) - \left[ \Lap_{0}, \chi_{3}\right]V(t) \right)\bigg).
  \end{equation}
  We claim that the term $\id + i \mathcal{F}^{-1}\left( \chi_{2}F(t) -
    \left[ \Lap_{0} , \chi_{3}\right] V(t)\right)$ is invertible in a
  logarithmic region, and so the estimates for $\chi R(\lambda) \chi$
  will follow from those of $\chi R^{\#}(\lambda) \chi$ in the same
  region.

  The following four estimates, proved below, will justify the above claim:
  \begin{align}
    \label{eq:E1}
    \norm[L^{2} \to \dom_{j}]{\chi \tilde{R}(\lambda)} &\leq C_{j}
    |\lambda|^{j-1} e^{T|\Im \lambda|}, \quad j = 0,1 \\
    \label{eq:E2}
    \norm[L^{2}\to L^{2}]{\widecheck{\chi F(\bullet)}} &\leq
    C_{j}|\lambda|^{-j} e^{T|\Im \lambda|}, \quad j = 0, 1, \ldots ,\lfloor s-2\rfloor\\
    \label{eq:E3}
    \norm[L^{2}\to \dom_{j}]{\chi \widecheck{(1-\chi_{3})V(\bullet)}}
    &\leq C_{j} |\lambda|^{j-1}e^{T|\Im\lambda|}, \quad j = 0,1 \\
    \label{eq:E4}
    \norm[L^{2}\to L^{2}]{\widecheck{\left[
          \Lap_{0},\chi_{3}\right]V(\bullet)}} &\leq C_{j}
    |\lambda|^{-j}e^{T|\Im\lambda|}, \quad j = 0, 1,\ldots, \lfloor s-2 \rfloor
  \end{align}
  
  Given these estimates, the theorem holds with $R(\lambda)$ replaced
  by $R^{\#}(\lambda)$ by \eqref{eq:E1}, \eqref{eq:E3}.  One may now take $\delta < T^{-1}$, and then
  there is some constant $C$ so that for $\Re \lambda > C$ and $\Im
  \lambda > - \delta \log \Re \lambda$, one has
  $C|\lambda|^{-1}e^{T|\Im\lambda|} < 1/4$, showing the invertibility
  of the claimed term in equation~\eqref{eq:resolvent-identity}.
  Shrinking $\delta$ then finishes the proof.  We must thus only prove
  the four estimates.

  The first two estimates follow from writing out the Fourier
  transform and noting that for $z$ in the support of $\chi$, there is
  some $T$ so that $\zeta (t,z) \equiv 0$ for $t > T$.
  Estimate~\eqref{eq:E1} for $j=1$ follows directly from the energy
  estimate, while the estimate for $j =0$ uses the energy estimate and
  integration by parts, as $\lambda e^{it\lambda} =
  \PD[t]e^{it\lambda}$.  The estimate~\eqref{eq:E2} follows by the
  same sort of integration by parts argument and the observation that
  $\chi F(t)$ is compactly supported in time.  The lack of smoothness
  in the $t$ variable prevents the estimate from holding for all $j$
  (and is one of the main differences of the set-up here from that using in \cite{Vainberg}).

  The other two estimates are somewhat more subtle and rely on
  properties of the free wave group.  We start by writing
  \begin{equation}
    \label{eq:V-identity}
    V(t) g = (1-\chi_{2})\zeta H(t) \Ut \chi g - H(t) \Utzero
    (1-\chi_{2})\chi g + q(t,z),
  \end{equation}
  where $\Utzero$ is the free sine propagator, i.e., $\Utzero = \sin
  t\sqrt{\Lap_{0}}/ \sqrt{\Lap_{0}}$.  By using the equation for $V$
  (and that $\Lap_{0} = P$ on $\supp (1-\chi_{3})$ and $\left[
    \PD[t]^{2},\zeta\right]$ is order $1$), we see that
\begin{equation}\label{foobar}
  \begin{aligned}
    \left( \PD[t]^{2} - \Lap_{0}\right)q &= - \left[ \Lap_{0},
      \chi_{2}\right]\zeta H(t) \Ut \chi g,  \\
    q(0,z) &= \PD[t]q(0,z) = 0.
  \end{aligned}
\end{equation}
  The inhomogeneous term in \eqref{foobar} has compact support
  in both space and time and vanishes identically for $t < 0$.  

We now undertake to show that \eqref{eq:E3}, \eqref{eq:E4} hold by
verifying them for each term on the right-hand-side of
\eqref{eq:V-identity}, starting with $q.$
  If $E_+$ is the forward fundamental solution for the wave equation
  on $\reals^{n}$, then we may write
  \begin{equation*}
    \chi (z) q(t,z) = - \int_{0}^{t} \chi(z) E_{+}(t-\tilde{t}) \star
    \left[ \Lap_{0}, \chi_{2}\right] \zeta \Ut[\tilde{t}] \chi g \, d\tilde{t}.
  \end{equation*}
  If $n$ is odd, then Huygens' principle implies that $\chi q$
  vanishes identically for large $t$.  Thus estimates \eqref{eq:E3},
  \eqref{eq:E4} hold with $V$ replaced by $q$ in this case..  If $n$ is even, then Huygens'
  principle no longer applies but if $t$ is large then the inhomogeneous
  term is disjoint from the singular support of $E_{+}$ and therefore
  for large $t$,
  \begin{equation*}
    \chi(z) q(t,z) = c_{n}\chi (z) \int_{0}^{\infty}
    \int_{\reals^{n}}\left( (t-\tilde{t})^{2} + |z-
      \tilde{z}|^{2}\right)^{\frac{1-n}{2}} \left( \left[ \Lap_{0},
        \chi_{2}\right]\zeta \Ut[\tilde{t}]\chi g(\tilde{z})\right)\,d\tilde{z}\,d\tilde{t}.
  \end{equation*}
  In particular, it is analytic for large $t$ and satisfies the
  derivative estimates \eqref{eq:E3}, \eqref{eq:E4} for large $t$ as well.

  A similar result holds for the second term in \eqref{eq:V-identity}, since
  \begin{equation*}
    \Utzero (1-\chi_{2})\chi g = E_{+}(t)\star (1-\chi_{2})\chi g.
  \end{equation*}

Now to see the estimate~\eqref{eq:E3}, first observe that the first term
  in equation~\eqref{eq:V-identity} does not occur because $\left[
    \Lap_{0}, \chi_{3}\right](1-\chi_{2}) = 0$.  Combining the
  estimates on $\chi q$ and $\Utzero$ for large $t$ with the known
  $C^{s-1}$ bounds on $V(t)$ we apply
  Lemma~\ref{lem:fourier-transform}.  

  The final estimate~\eqref{eq:E4} follows in the same manner, but
  uses the estimate~\eqref{eq:E1} to bound the first term from
  equation~\eqref{eq:V-identity}.  
\end{proof}

\appendix
\section{Proof of Proposition~\ref{proposition:jacobi}}
\label{sec:jacobi-proof}

In this section we prove Proposition~\ref{proposition:jacobi}.  By
Remark~\ref{remark:jacobi} it suffices to prove it for
$\flowoutX_{\alpha}$ and $\flowoutX_{\beta}$ rather than
$\flowout_{\alpha}$ and $\flowout_{\beta}$.

In what follows, $\phi_{t}: TX\to TX$ denotes geodesic flow for
time $t$ and $\gamma_{(p,v)}$ denotes the geodesic in $X$ such that
$\gamma(0) = p$ and $\gamma'(0) = v$.  We denote by $\pi$ the
projection $TX \to X$.  

The kernel of the pushforward $\pi_{*}$ consists of those vectors in
$T_{(p,v)}(TX)$ tangent to the fibers of $TX$ and is referred to as
the \emph{vertical subspace} of $T_{(p,v)}(TX)$.  The Levi-Civita
connection on $X$ defines the \emph{connection map} $K: T_{(p,v)}(TX)
\to T_{p}X$, whose kernel defines a horizontal subspace of
$T_{(p,v)}(TX)$.  The connection map $K$ thus provides an
identification of the vertical subspace of $T_{(p,v)}(TX)$ with
$T_{p}X$.  Similarly, the pushforward $\pi_{*}$ provides an
identification of the horizontal subspace with $T_{p}X$.  

We now recall the following characterization of Jacobi fields on a
Riemannian manifold $X$ from Proposition~1.7 of
Eberlein~\cite{Eberlein}.
\begin{lemma}
  \label{lem:eberlein}
  If $\gamma_{(p,v)}$ is a geodesic in $X$ then there is a one-to-one
  correspondence between Jacobi fields along $\gamma$ and vectors in
  $T_{(p,v)}(TX)$.  In particular, for $\zeta \in T_{(p,v)}(TX)$, let
  $Y_{\zeta}(t)$ be the unique Jacobi field along $\gamma$ with
  $Y_{\zeta}(0) = \pi_{*}\zeta$ (the horizontal part of $\zeta$) and
  $Y_{\zeta}'(0) = K\zeta$ (the vertical part of $\zeta$).  Then
  $Y_{\zeta}(t)$ is given by the following:
  \begin{equation*}
    Y_{\zeta}(t) = \pi_{*}(\phi_{t})_{*}\zeta, \quad Y_{\zeta}'(t) = K (\phi_{t})_{*}\zeta.
  \end{equation*}
\end{lemma}
In other words, the Jacobi field and its derivative are, taken
together, invariant under the geodesic flow. 

Note further that because the metric is flow-invariant, if
$\pi_{*}\zeta$ and $K\zeta$ are both orthogonal to $v\in T_{p}X$, then
the Jacobi field $Y_{\zeta}(t)$ is everywhere orthogonal to
$\gamma'(t)$.  

We now abuse notation and consider $\flowoutX_{\alpha}^{s}$ and
$\flowoutX_{\beta}^{t}$ as subsets of $TX$ rather than $T^{*}X$.  We
further abuse notation and use $\xi$ and $\eta$ to denote coordinates
on the \emph{tangent bundle} (with $\xi \pa_x+\eta \cdot \pa_y$
denoting the corresponding tangent vector) rather than the cotangent bundle.  We
note further that they extend to the boundary and use a bar to denote
their completions, i.e., in a neighborhood $U$ of $Y_{\alpha}$, 
\begin{equation*}
  \overline{\flowoutX^{s}_{\alpha}} = U \cap \{ 0 \leq x < s,\ y
  \in Y_{\alpha},\ \xi \in \reals,\ \eta = 0\}.
\end{equation*}
Observe that if $s > 0$ and $\xi > 0$, then $\phi_{s}(0,y,\xi,0 ) =
(s\xi, y, \xi, 0)$.  A similar statement holds for $s, \xi < 0$.  

We also require the following lemma, which describes how the
pushforward acts on the tangent space to $\flowoutX_{\alpha}^{s}$.
\begin{lemma}
  \label{lem:pushforward}
  Suppose $(p,v) = (x,y,\xi,0) \in \flowoutX^{s}_{\alpha}$ lies in a
  small neighborhood of $Y_{\alpha}$ and that $\zeta \in
  T_{(p,v)}(\flowoutX^{s}_{\alpha})\subset T_{(p,v)}(TX)$.  If
  $K\zeta$ and $\pi_{*}\zeta$ are both orthogonal to $v$ then $K\zeta
  = 0$ and $\pi_{*}\zeta = \tilde{\zeta}' \cdot \pd[y]$.

  Moreover, for small $t$ (i.e., $t$ so that $x + t \xi \geq 0$ and
  $\phi_{t} (p,v)$ still lies in this small neighborhood), then
  \begin{equation*}
    K(\phi_{t})_{*}\zeta  = 0, \quad \pi_{*}(\phi_{t})_{*}\zeta =
    \tilde{\zeta}'\cdot \pd[y].
  \end{equation*}
\end{lemma}

The proof of Lemma~\ref{lem:pushforward} is a simple calculation and
omitted here for brevity.  

We also remark that because $\flowoutX_{\alpha}^{s}$ is given as a
flow-out, if $\zeta$ is tangent to this flow-out then so is
$(\phi_{t})_{*}\zeta$.  

We now turn our attention to the proof of Proposition~\ref{proposition:jacobi}.
\begin{proof}[Proof of Proposition~\ref{proposition:jacobi}]
  Observe that if $\flowoutX_{\alpha}^{s}$ and
  $\flowoutX_{\beta}^{t}$ intersect then a segment of a geodesic
  connecting $Y_{\alpha}$ and $Y_{\beta}$ lies in their intersection,
  i.e., $\flowoutX_{\alpha}^{s}\cap \flowoutX_{\beta}^{t}$
  contains a segment of a geodesic $\gamma$ with $\gamma(0) \in
  Y_{\alpha}$ and $\gamma(t') \in Y_{\beta}$ with $t'< s+t$.  

  We start by assuming that $\flowoutX_{\alpha}^{s}$ and
  $\flowoutX_{\beta}^{t}$ intersect non-transversely at the point
  $(p,v) \in TX$.  We let $\gamma$ be the geodesic through $(p,v)$ and
  observe that $\gamma$ connects $Y_{\alpha}$ and $Y_{\beta}$.  Both
  flow-outs have dimension $n+1$ and the tangent space to their
  intersection contains both the direction of the flow and the radial
  vector field.  Because the intersection is non-transverse, it also
  contains a vector $\zeta$ linearly independent of the previous two.
  We may thus assume that both $\pi_{*}\zeta$ and $K\zeta$ are
  orthogonal to $v$.  By pushing $\zeta$ forward by the flow and
  applying Lemma~\ref{lem:eberlein}, we obtain a family of vectors
  tangent to both flow-outs that corresponds to a normal Jacobi field
  along $\gamma$.  Lemma~\ref{lem:pushforward} then implies that this
  family has a limit at the cone points that projects to a b-vector
  field, since the coefficient of $\pa_x$ vanishes at the boundary.
  This completes one direction of the proof.

  Conversely, suppose that $W$ is a Jacobi field along $\gamma$ so that $W(0)$
  and $W(t')$ are both b-vector fields.  By standard arguments we may
  assume that $W$ is normal to the flow.  By using
  Lemma~\ref{lem:eberlein}, we may associate to $W$ a vector field
  $\zeta$ along $\gamma$ in $T(TX)$.  Because it is a b-vector field
  at the endpoints, Lemma~\ref{lem:pushforward} implies that it is
  initially tangent to both flow-outs and so stays tangent under the
  pushforward.  Because $\pi_{*}\zeta$ and $K\zeta$ are both
  orthogonal to the flow, the tangent space of the intersection along
  this geodesic must have dimension at least $3$.  Each flow-out has
  dimension $n+1$ and therefore a transverse intersection must have
  dimension $2$, implying that the intersection is non-transverse.
\end{proof}

\enlargethispage{\baselineskip}
\end{document}